\documentclass{amsart}
\usepackage{amsmath,amsthm,amssymb}
\usepackage{mathtools}
\usepackage{enumitem}
\usepackage{amsrefs}

\usepackage[colorlinks=true,citecolor=blue]{hyperref}

\newtheorem{thm}{Theorem}[section]
\newtheorem{cor}[thm]{Corollary}
\newtheorem{lem}[thm]{Lemma}
\newtheorem{prop}[thm]{Proposition}
\newtheorem{claim}[thm]{Claim}
\newtheorem{fact}[thm]{Fact}

\theoremstyle{definition}
\newtheorem{defn}[thm]{Definition}

\theoremstyle{remark}
\newtheorem{ntn}[thm]{Notation}
\newtheorem{rem}[thm]{Remark}

\newtheorem{sample}[thm]{Example} \numberwithin{equation}{section}
\numberwithin{equation}{section}

\def\GL{\mathrm{GL}}
\def\SL{\mathrm{SL}}
\newcommand{\RR}{\mathbb{R}}

\newcommand{\CN}{\mathcal{N}}

\newcommand{\rarr}{\rightarrow}

\newcommand{\la}{\langle}
\newcommand{\ra}{\rangle}
\def\tp{\mathrm{tp}}

\newcommand{\sub}{\subseteq}
\def\CL{\mathcal{L}}
\def\st{\operatorname{st}}
\def\esub{\prec}
\def\sube{\succ}
\def\CM{\mathcal{M}}
\def\CO{\mathcal{O}}
\def\UU{\mathbb{U}}
\def\Smu{\operatorname{Stab}^\mu}
\def\Stab{\operatorname{Stab}}
\def\ccdot{{\cdot}}
\def\SO{\mathrm{SO}}

\DeclareRobustCommand{\rchi}{{\mathpalette\irchi\relax}} % raised \chi
\newcommand{\irchi}[2]{\raisebox{0.35ex}{$#1\chi$}} % inner command, used by \rchi

\newcommand*\bbar[1]{% nicer looking \bar
  \hbox{%
    \vbox{%
      \hrule height 0.5pt % The actual bar
      \kern0.3ex%         % Distance between bar and symbol
      \hbox{%
        \kern-0.1em%      % Shortening on the left side
        \ensuremath{#1}%
        \kern-0.1em%      % Shortening on the right side
      }%
    }%
  }%
}
\newcommand{\al}{\alpha}
\newcommand{\bH}{\bbar{\mathbb{H}}}

\begin{document}

% \baselineskip=17pt
% \titlerunning{{\large $\mu$}-types and stabilizers}

\title[{\large $\mu$}-types and stabilizers]{Topological groups, {\large $\mu$}-types and their stabilizers}

\author{Ya'acov Peterzil}
\address{University of Haifa}
\email{kobi@math.haifa.ac.il}
\author{Sergei Starchenko}
\address{University of Notre Dame}
\email{sstarche@nd.edu}

% \author{Ya'acov Peterzil
% \and
% Sergei Starchenko
% }

% \date{November 3, 2014}

% \maketitle

% \address{Y.~Peterzil: University of Haifa;  \email{kobi@math.haifa.ac.il}
% \and
% S.~Starchenko: University of Notre Dame;  \email{sstarche@nd.edu}
% }


\maketitle

\begin{abstract} We consider an arbitrary topological group $G$ definable in a structure $\CM$, such
that some basis for the topology of $G$ consists of  sets definable in $\CM$.
 To each such group $G$ we associate a compact $G$-space of partial types
$S^\mu_G(M)=\{p_\mu\colon p\in S_G(M)\}$ which is the quotient of the usual type
space $S_G(M)$ by the relation of two types being ``infinitesimally close to each
other''. In the o-minimal setting, if $p$ is a definable type then it has a
corresponding definable subgroup $\Smu(p)$, which is the stabilizer of $p_\mu$. This
group is nontrivial when $p$ is unbounded; in fact it is a torsion-free solvable
group.

Along the way, we analyze the general construction of $S^\mu_G(M)$ and its
connection to the  Samuel compactification of topological groups.
\end{abstract}

\section{Introduction} It was shown in \cite{PS} that in a  group $G$
definable in an
 o-minimal structure, one can associate to any definable unbounded curve $\gamma\sub G$ a
definable one-dimensional torsion-free group $H_{\gamma}$. In fact, the group
$H_\gamma$ can be viewed as associated to the (definable) type $p$ of $\gamma$ at
``$+\infty$''. Our initial goal in the current article  was to extend that result to
arbitrary definable types in $G$ and associate to any such $p$ a definable group
$H_p$, which is nontrivial if and only if $p$ is unbounded (here, a type $p$ is
called {\em unbounded\/} if no formula in $p$ defines a definably compact set with
respect to the $G$-topology).

While working on the above  we discovered  interesting  connections to general
topological groups, $G$-spaces and their universal compactifications. Namely,
consider an arbitrary topological group $G$, definable in some structure $\CM$, with
a basis for its topology consisting of  sets definable in $\CM$. Under these
assumptions we view the partial type $\mu$ of all definable open subsets of $G$
containing the identity as an ``infinitesimal subgroup'' and use it to define an
equivalence relation on complete types in $S_G(M)$: $p\sim_\mu q$ if $\mu\ccdot
p=\mu\ccdot q$, as partial types. It turns out that this equivalence relation is
well behaved and the quotient space $S^\mu_G(M)$ is a compact $G$-space. Moreover,
this construction recovers the  Samuel compactification, see \cite{samuel}, in the
case of an arbitrary topological group, when {\em all\/} subsets of $G$ are definable
in $\CM$. In the case when $G$ is a discrete group our analysis is already subsumed
by the work of Newelski \cite{N} and others (see for example \cite{GPP}).

Returning to our original problem, the group $H_p$ described above is just the
stabilizer of the associated partial type $p_\mu$ under the action of $G$ on
$S^\mu_G(M)$. For the result below, we say that a complete type $p\in S_G(M)$ is
$\mu$-reduced, if its $\sim_\mu$-class does not contain any complete type of lower
dimension. We first prove (see Claim~\ref{claim:def-red}) that any definable type
$q$ which is unbounded has a $\sim_\mu$-equivalent definable type $p$ of positive
dimension which is $\mu$-reduced. Summarizing our main results we have
(see theorems
\ref{Prop:definable} and \ref{theorem:1}):

\begin{thm}\label{intro1} Let $G$ be a definable group in an o-minimal expansion of a real
closed field. Then to any definable $\mu$-reduced type $p\in S_G(M)$ there is an
associated definable, torsion-free  group $H_p=\Smu(p)\supseteq \Stab(p)$, with
$\dim H_p=\dim p$. In particular, $\dim H_p>0$ if and only if $p$ is unbounded.
\end{thm}

Regarding the general setting of a topological group $G$ which is definable in a
structure $\CM$ and has a basis of definable sets, we prove in Appendix A the
following (see Theorem~\ref{thm-def-sam}):
\begin{thm} The quotient $S^\mu_G(M)$ is a compact Hausdorff space on which the group $G(M)$
acts continuously. The map $g\mapsto {tp(g/M)}_\mu$ embeds $G(M)$ as a dense subset of
$S^\mu_G(M)$.
\end{thm}
Thus, $S^\mu_G(M)$ is a $G$-ambit, and we show that it is the greatest $G$-ambit
within the so-called definably separable $G$-ambits.

\medskip

While our main interest is with the group $G$ itself we found it useful to treat the
more general case of  $G$ acting definably on a definable set $X$, and this is the
setting throughout the first part of the paper.
\medskip

 \noindent {\bf A uniformity vs.\ a type-definable
equivalence relation.}   As we pointed out, our construction of $S^\mu_G(M)$
recovers to a great extent the work of Samuel on compactifications of uniform spaces
(see \cite{samuel}). Samuel works under the assumption of a set together with a {\em
uniformity}, a collection of subsets of $X^2$ which satisfies certain conditions and
gives rise to a topology on $X$. As we explain in Appendix A, the notion of a
uniformity is the same as the model theoretic notion of a type-definable equivalence
relation on $X$. Thus, we carry out some of the work in this more general setting.
\medskip

\noindent {\bf Outline of the paper}  We begin  in Section 2  by a  discussion of
topological groups and $\mu$-types in general. In Section 3 we move to the o-minimal
setting and analyze in details the group $\Smu(p)$. In Appendix A we carry out the
above mentioned analysis of the general case and describe the space $S^\mu_G(M)$ in
details. In Appendix B we prove a technical result which is needed for the o-minimal
case.
\medskip

\noindent{\bf General conventions} We fix a complete first order theory $T$ and a
large saturated enough model $\UU$ of $T$.   When $D$ is a definable set and
$\varphi(x)$  a formula, we say that $\varphi(x)$ is a {\em $D$-formula\/} if
$\UU\models \varphi(x)\rarr x\in D$. We extend this definition to types (possibly
incomplete) by saying that a type $q(x)$ is a \emph{$D$-type} if $q(x)\vdash x\in
D$. We let $\mathcal L_D(M)$ be the collection of all $D$-formulas with parameters
in $M$ and let $S_D(M)$ denote the set of all complete $D$-types over $M$.

We  use the predicate  $D$ to denote both a definable set and a formula defining
it. Thus, we use for example $g\in D$ or $p\vdash D$, where $D$ is thought of as a
definable set in the first case and a formula in the latter. When $D$ is a definable
set over a model $\CM$ and $\CN\succ \CM$ then we will write $D(M)$ or $D(N)$ to denote the
specific realization of $D$ in the structures $\CM$ or $\CN$.
Very often when we work in a fixed structure $\CM$ we just write $D$
instead of $D(M)$.

\section{Topological groups and $\mu$-types}
\label{sec:prelim}

\subsection{On definable groups and group actions}\label{sec:group-actions}

Let $G$ be a group. Recall that a {\em $G$-set\/} is a set $X$ together with an action
of $G$ on it. If $X$ is a $G$-set then for subsets $P\subseteq G$, $Y\subseteq X$,
we will denote by $P\ccdot Y$ the set $\{ x{\ccdot} y \colon x\in P, y\in Y
\}\subseteq X$, and if $P$ (or $Y$) is a singleton $\{a\}$ then we just write
$a\ccdot Y$ ($P\ccdot a$).

When $G$ is a topological group, $X$ is a topological space and the action is
continuous, (i.e.\  the map $(g,x)\mapsto g\ccdot x$ is a continuous function from
$G\times X$ to $X$), then $X$ is called {\em a $G$-space}.
 All topological groups and compact spaces are assumed to be
Hausdorff.

\medskip
{\em Throughout Section~\ref{sec:prelim} we fix a small arbitrary $\CM\esub \UU$ and
an $M$-definable group $G$. We also fix an $M$-definable set $X$ with $G$ acting
definably on $X$, namely, the map $(g,x)\mapsto g\ccdot x$ is definable in $\CM$. We
call $X$ a definable $G$-set.}

\begin{ntn}
  \begin{enumerate}[leftmargin=*]
  \item   If $\varphi(v)$ is a   $G$-formula
over $M$ and $\psi(x)$ is an $X$-formula then  by $\varphi\ccdot \psi$ we will
denote the $X$-formula
\[  (\varphi\ccdot \psi)(x) = \exists v\exists u \,\bigl( \varphi(v) \,\&\,
\psi (u) \,\&\, x=v\ccdot u \bigr). \]
\item  If $p(v)$ is a $G$-type and $r(x)$ is an $X$-type  (possibly incomplete) then
by $p\ccdot r$ we will denote the $X$-type
\[ (p\ccdot r)(x) =\{ (\varphi\ccdot \psi)(x) \colon p(v)\vdash\varphi(v)
, r(x)\vdash \psi(x) \}. \]
 \end{enumerate}
\end{ntn}

\begin{rem}\label{rem:products}
\begin{enumerate}[leftmargin=*]
\item Considering the action of $G$ on itself by left multiplication the above
definition gives us a notion of products of types,
  but for $p,q\in S_G(M)$  the type $p\ccdot q$ is usually
  incomplete.
\item Identifying an element $g\in G(M)$ with the complete type
  $\tp(g/M)$ the above definition agrees with the usual definition of
  the action of $G(M)$ on the $G$-types and the $M$-definable $G$-sets.
\item It is easy to see that if $\varphi(v)$ is a $G$-formula over $M$
  and $\psi(x)$ is an $X$-formula over $M$ then
\[ \varphi(M)\ccdot \psi(M) = (\varphi\ccdot \psi) (M). \]
\item  For a $G$-type $p(v)$ over $M$ and an $X$-type $r(x)$ over $M$
  we have
\[  p(\UU)\ccdot r(\UU)=(p\ccdot r)(\UU). \]
But if $\CN\sube \CM$ is not $|M|^+$-saturated then in general we have only
inclusion
 \[  p(N)\ccdot r(N) \subseteq (p\ccdot r)(N). \]
\item  It follows from (4) that if $p(v), q(v)$ are $G$-types over $M$
  and $r(x)$ is an $X$-type over $M$ then
\[ \bigl( (p\ccdot q)\ccdot r\bigr)(x) =  \bigl(p\ccdot (q\ccdot r)\bigr)(x). \]

\end{enumerate}
\end{rem}

\subsection{Topological groups and their infinitesimal types}

For the rest of Section~\ref{sec:prelim} we assume in addition that $G$ is a
topological group and furthermore that $G$ has a basis for its topology consisting
of sets definable in $\CM$. Note that this is a rather weak assumption and for
example does not imply that $G(\mathbb U)$ is still a topological group. When we
develop the theory further, we make a stronger assumption, that a basis for the
topology of $G$ is {\em uniformly definable\/} in $\CM$ (which is what Pillay calls in
\cite{pillay-t} ``a first order topological group''). This will be sufficient to
ensure that $G(\mathbb U)$ is a topological group.

\begin{defn}  The {\em infinitesimal type\/} of $G$
is the partial $G$-type over $M$, denoted by $\mu_G(v)$ (or just by $\mu(v)$ if $G$
is fixed), consisting of all formulas over $M$ defining  an  open neighborhood of
$e$.
\end{defn}

Notice that the type $\mu(v)$ is not complete unless the topology on $G$ is
discrete. The next claim follows from the continuity of the group operations.

\begin{claim}\label{claim:cont}
  \begin{enumerate}[leftmargin=*]
  \item  $\mu(v)=\mu^{-1}(v)$.
  \item For every $g\in G(M) $ we have $g\ccdot \mu = \mu\ccdot g$.
  \item $\mu\ccdot\mu =\mu$.
  \end{enumerate}
\end{claim}

\begin{cor}\label{mu-group}
For any elementary extension $\CN$ of $\CM$ the set $\mu(N)$ is a subgroup of $G(N)$ and
every element of $G(M)$ normalizes $\mu(N)$.
\end{cor}

\begin{claim}\label{claim:inf-eq2} For a partial $X$-type $\Sigma(x)$
over $M$ and $p\in S_X(M)$, the following are equivalent:
\begin{enumerate}
\item $(\mu\ccdot p) \cup \Sigma$ is consistent.

\item $p\vdash \mu\ccdot \Sigma$.

\end{enumerate}
\end{claim}
\begin{proof}  We work in $\mathbb U$.

$1\Rightarrow 2$: Assume that $\mu\ccdot p\cup \Sigma$ is consistent and fix in
$G(\mathbb U)$, $X(\mathbb U)$, elements $\epsilon\models \mu$, $b\models p$,
respectively, such that $\epsilon\ccdot b\models \Sigma$. Let $\beta=\epsilon\ccdot
b$. Since $b=\epsilon^{-1}\beta$ and $\mu^{-1}=\mu$ we have $b\models \mu\ccdot
\Sigma$. Since $p$ is a complete type, it implies $p\vdash \mu\ccdot \Sigma$.

$2\Rightarrow 1$: Assume that $p\vdash \mu\ccdot \Sigma$ and choose $\epsilon\models
\mu$ and $b\models p$ such that $\epsilon^{-1}\ccdot b\models \Sigma$. Since
$\mu^{-1}=\mu$, $\epsilon^{-1}\ccdot b\models \mu\ccdot p$ so the result
follows.\end{proof}

We now conclude:

\begin{claim}\label{claim:inf-eq}
  For $p, q\in S_X(M)$, the following conditions are
  equivalent.
  \begin{enumerate}
  \item  The type $(\mu \ccdot p)(x) \cup (\mu\ccdot q)(x)$ is
    consistent.
\item $p(x)\vdash (\mu\ccdot q)(x)$. \item $\mu\ccdot p = \mu\ccdot q$ (here and
below we consider two partial types over $M$ to be equal if they are logically
equivalent).

  \end{enumerate}
\end{claim}
\begin{proof} We apply Claim~\ref{claim:inf-eq2}, by taking $\Sigma= \mu\ccdot q$, and using
$\mu\ccdot \mu=\mu$.\end{proof}

\begin{ntn} \leavevmode
\begin{itemize}[label=\textbullet, leftmargin=*]

 \item For
$p,q\in S_X(M)$ write $p\sim_\mu q$ if $\mu\ccdot p=\mu\ccdot q$ as partial types.

\item  We let $S_X^\mu(M)$ be the quotient of $S_X(M)$ by the equivalence relation
$\sim_\mu$, and use $p_\mu$ to denote the $\sim_\mu$-equivalence class of a type
$p$. Namely,  $\mu\ccdot p$ is  a partial type and $p_\mu$ is the associated
equivalence class.

\end{itemize}
\end{ntn}

\begin{claim}
\label{claim:g-action} For any $g\in G(M)$ and $p\in S_X(M)$ we have
\[ g\ccdot (\mu\ccdot p)= \mu\ccdot(g\ccdot p). \]
\end{claim}
\begin{proof} Follows from Remark~\ref{rem:products}(5) and  Claim~\ref{claim:cont}(2). \end{proof}

Thus the action of $G(M)$ on $S_X(M)$ preserves $\sim_\mu$, so it induces an action
of $G(M)$ on $S^\mu_G(M)$ by $g\ccdot p_\mu={(g\ccdot p)}_\mu$. We will consider
$S^\mu_X(M)$ as a $G$-set. In Appendix A we discuss other properties of
$S_X^\mu(M)$.

\subsection{$\mu$-stabilizers}

\leavevmode

 We still assume that $X$ is a definable $G$-set (and $X$ not
assumed to carry any topology).

\subsubsection{Stabilizers of partial types}

As we pointed out already, the group $G(M)$ acts on $\mathcal L_X(M)$.
\begin{defn} Let $\Sigma(x)$ be a partial $X$-type  over $M$. We first define
$$\Stab(\Sigma)=\{g\in G(M)\colon \mbox{ for all } \varphi\in \mathcal L_X(M),\,\,
\Sigma \vdash \varphi \Leftrightarrow \Sigma \vdash g\ccdot \varphi\}.$$ For
$\varphi(x)\in \mathcal L_X(M)$, consider the set
\begin{equation}\label{set}\{h\in G(M)\colon \Sigma\vdash h\ccdot \varphi(x).\}\end{equation} and define $
\Stab_{\varphi}(\Sigma)\sub G(M)$ to be the stabilizer of the above set (so in
particular a subgroup).
\end{defn}

The following is easy to verify:
\begin{claim}\label{claim:stab1} For every partial $X$-type $\Sigma$
  over $M$,
$$\Stab(\Sigma)=\bigcap_{\Sigma \vdash \varphi} \Stab_{\varphi}(\Sigma)=
\bigcap_{\varphi\in \mathcal L_X(M)} \Stab_{\varphi}(\Sigma).$$
\end{claim}

\begin{defn} We say that a partial type $\Sigma(x)\sub \mathcal L(M)$ is {\em
definable\/} over $A\sub M$ if for every formula $\phi(x,y)$ there
exists a formula $\rchi(y)\in \mathcal L(A)$ such that for every $a\in M$,
$\Sigma\vdash \phi(x,a)$ if and only if $M\models \rchi(a)$.

For $\Sigma(x)$ a partial type definable over $M$, and $\mathcal N\succ \CM$ we
denote by $\Sigma|N$ the extension of $\Sigma$ by definitions to a partial type
over $N$. Namely, for $a\in N$, and $\phi\in \mathcal L(M)$, $\phi(x,a)\in
\Sigma|N$ iff $\CN\models \rchi(a)$, for $\rchi(y)$ as above.
\end{defn}
\vspace{.3cm}

Here is our main use of definability of types:
\begin{prop}\label{claim:stab2} Assume that $\Sigma$ is a definable partial $X$-type over $M$.
Then $\Stab(\Sigma)$ can be written as the intersection of $M$-definable subgroups.

If in addition $G$ has the Descending Chain Condition on $M$-definable subgroups
then $\Stab(\Sigma)$ is a definable subgroup of $G$.
\end{prop}
\begin{proof} The fact that each of the sets in \eqref{set} is definable is immediate from
the definability of $\Sigma$. It follows that each $\Stab_\varphi(\Sigma)$ is
definable and therefore $\Stab(\Sigma)$ is the intersection of definable groups. If
in addition $G$ has DCC, then the intersection is finite hence definable.\end{proof}

\medskip

\noindent{\bf Strengthening the assumptions}

\noindent  {\em From now on we assume that $G$ has a uniformly definable basis
$$\{B_t\colon t\in T\}$$ of open neighborhoods of the identity. We call such $G$  a definably topological
group. As pointed out earlier, for $\CN\succ\CM$ the group $G(N)$ is again a
topological group and the definable family $\{B_t\colon t\in T(N)\}$ forms a basis for the
open neighborhoods of $e$.}

\medskip

We may identify the type  $\mu$ with the collection of formulas $\{B_t\colon t\in
T\}$. Note that $\mu$ itself is a definable partial type, over the parameters
defining $T$. Indeed, for $a\in M$, $\mu\vdash \phi(x,a)$ if and only if $\CM\models
\exists t\in T \, \forall x\,(x\in B_t\to \phi(x,a))$. If $\CN\succ \CM$ then $\mu|
N$ is just the infinitesimal type of $G(N)$ in the structure $\CN$.

\subsubsection{Stabilizers of $\mu$-types}

\begin{ntn}  For $p\in S_X(M)$,
we define \emph{the infinitesimal stabilizer $\Smu(p)$} as
\[   \Smu(p) =\Stab(\mu \ccdot p).\]
\end{ntn}

Note that $\Smu(p)$ contains the usual stabilizer of $p$, denoted by $\Stab(p)$. The
claim below and its proof was proposed to us by A.~Pillay.
\begin{claim}\label{claim:stab3}
 If $p$ is a complete type in $S_X(M)$ definable over $A\sub M$ then $\mu\ccdot p$
is a partial type definable over $A$.
\end{claim}
\begin{proof}
  Let $\phi(x,y)$ be a formula. For $b\in M$ we have
$(\mu\ccdot p)(x)\vdash \phi(x,b)$ if and only if $(\mu\ccdot p)(x) \cup\{\neg
\phi(x,b)\}$ is inconsistent, that  by Claim~\ref{claim:inf-eq2} is equivalent to
$p(x)\not\vdash (\mu\ccdot \neg\phi(x,b))(x)$. Since $p$ is a complete type, the
latter condition is equivalent to the existence of $t\in T(M)$ such that $p(x)\vdash
\neg (B_t\ccdot \neg \phi(x,b))(x)$.

Since the type $p(x)$ is definable over $A$, there is a formula $\rchi(u,y)\in\CL(A)$ such that
$p(x)\vdash \neg (B_t\ccdot \neg \phi(x,b))(x)$ if and only if
$\CM\models \rchi(t,b)$.

Thus $(\mu\ccdot p)(x)\vdash \phi(x,b)$ if and only if $\CM\models
\exists t\in T
(\rchi(t,b))$.\end{proof}

\begin{rem} The definition of $\mu\ccdot p$ is canonical in the following sense: If $p\in S_X(M)$
is definable over $M$ and $\CN\succ \CM$ then $(\mu| N)\ccdot (p| N)=(\mu\ccdot
p)| N$.\end{rem}

\begin{prop}\label{DCC} Assume that
$G$ has the Descending Chain Condition for definable subgroups and that  $p\in
S_X(M)$ is a definable type over $M$.  Then $\Smu(p)$ is an $M$-definable subgroup
of $G$.

Moreover, if $\CN\succ \CM$ then $\Smu(p| N)$ is  defined by the same formula.
\end{prop}
\begin{proof}  We take $\Sigma=\mu\ccdot p$ which is definable by the last claim. By Claim~\ref{claim:stab2}, $\Smu(p)$ is definable in $\CM$.\end{proof}

The following claim follows  from generalities of group actions.
\begin{claim}\label{claim:conjugate} For $p,q\in S_X(M)$ and $g\in G(M)$,
if $g\ccdot p_\mu = q_\mu$ then $\Smu(q)=g\Smu(p)g^{-1}$.\end{claim}

\begin{rem} In the case $X=G$
we always consider the action of $G$ on itself by left multiplication. In
particular, both stabilizer groups $\Stab(p)$ and $\Smu(p)$  are taken with respect
to left multiplication. Note that they could turn out to be different groups for the
opposite action. The following example is taken from  \cite{gp}.

Let $G=\SL(2,\RR)$. Consider the curve
\[\gamma(t) =
\begin{pmatrix}
t & 1 \\
0 & t^{-1}
\end{pmatrix}, t>0. \]

We will denote by $S\subseteq G$ the image of $\gamma$, and let $p(x)$ be the type
on $S$ corresponding to $t>\mathbb R$.

It is not hard to see that the $\mu$-stabilizer of $p$ with respect to left
multiplication is:  \[ \left \{ \begin{pmatrix}
1 & a \\
0 & 1
\end{pmatrix}, a\in \mathbb R\right \}\]
and with respect to right-multiplication is:\[  \left \{
\begin{pmatrix}
r & 0 \\
0 & r^{-1}
\end{pmatrix}, r\in \mathbb R^{>0}\right \}. \]

\end{rem}

\subsection{A partial ``standard part'' map}
In the o-minimal case, when $\CN$ is a tame extension of $\CM$, there is a standard part
map $\st\colon \CO_M(N)\to M$ from the set of $M$-bounded elements of $N$ into $M$ (see
Section~\ref{sec:standard-part-map} for more details).  Here we introduce an analogue of
a standard part map without assuming o-minimality or tameness.

Let $\CN$ be an elementary extension of $\CM$.  By Corollary~\ref{mu-group},
$\mu(N)$ is a subgroup of $G(N)$ normalized by $G(M)$, hence  $\mu(N)\ccdot G(M)$ is
a subgroup of $G(N)$, still normalized by $G(M)$. In abuse of notation we denote
this subgroup by $\mathcal O_G(N)$.

Because $\mu(N)\cap G(M)=\{e\}$ there exists a surjective group homomorphism
$\st^*\colon \mathcal O_G(N)\to G(M)$ which sends every $h$ to the unique $g\in
G(M)$ such that $h\in \mu(N)g$. We think of the map $\st^*$ as a partial ``standard
part'' map from $G(N)$ into $G(M$), so for $Y\sub G(N)$ we will use $\st^*(Y)$,
again in abuse of notation, to denote the image under $\st^*$ of the set $Y\cap
\mathcal O_G(N)$.

Note that if $H\leq G(N)$ is any subgroup then  $\st^*(H)$ is a subgroup of $G(M)$.
As we shall see later, under various assumptions this group is definable in $M$.

\begin{sample}
If $G=\la \mathbb R,+\ra$ and $\mathcal R$ is a nonstandard
elementary extension of the real field then $\mathcal O_G(\mathcal R)$ equals the
subgroup of elements of $G$ of finite size. If $G=\la \RR^{>0},\ccdot\ra$ then
$\mathcal O_G(\mathcal R)$ is the collection of all positive elements $a\in \mathcal
R$ such that both $a$ and $1/a$ are finite.
\end{sample}

\begin{claim}\label{normal} Let $p, q\in S_X(M)$, $\CN\sube \CM$ be
  $|M|^+$-saturated
 and $\alpha\models p$, $\beta\models
q$ be in $N$.  For $a\in X(N)$, let $G_a=\{h\in G(N)\colon h\ccdot a=a\}$. Then,

(1) If $p\sim_\mu q$ then $\st^*(G_{\alpha})=\st^*(G_\beta)$.

(2) The group $\st^*(G_\alpha)$ is normal in $\Smu(p)$.

\end{claim}
\begin{proof}Let us see (1). Note that if $p=q$ then we can immediately conclude that
$\st^*(G_\alpha)=\st^*(G_\beta)$. Our result shows that it is sufficient to assume
that $p$ and $q$ are $\mu$-equivalent types.

 Because $p\sim_\mu q$ we may replace $\beta$ by
another realization of $q$ if needed and assume $\beta=\epsilon \ccdot\alpha$ for some
$\epsilon\in \mu(N)$. It follows that $G_{\beta}=\epsilon G_\alpha
\epsilon^{-1}$. Also, it is not hard to see that  $G_\beta\cap \mathcal O_G(N)=\epsilon (G_{\alpha}\cap \mathcal O_G(N))\epsilon^{-1}$. Because $\st^*$ is
a homomorphism  and $\epsilon\in \ker(\st^*)$, we see that
$\st^*(G_\beta)=\st^*(G_\alpha)$.

To see (2) we first note that the group $\st^*(G_\alpha)$ is contained in $\Smu(p)$.
Indeed, assume that $g=\st^*(h)$ for some $h\in G_{\alpha}$, so $g=\epsilon h$ for $
\epsilon\in \mu(N)$. Then,
$$g\ccdot \alpha=\epsilon h\ccdot \alpha=\epsilon\ccdot \alpha\models \mu\ccdot p,$$ so $g\in
\Smu(p)$.

To see that $\st^*(G_\alpha)$ is normal in $\Smu(p)$, we take $g\in \Smu(p)$ (in
particular, $g\in G(M)$), and then
$$g(\st^*(G_\alpha))g^{-1}=\st^* (gG_\alpha g^{-1})=\st^*(G_{g\ccdot \alpha})=\st^*(G_\alpha),$$ where the last
equality follows by (1), since $g\ccdot \alpha\sim_\mu \alpha$.  Thus we showed that
$ \st(G_\alpha)$ is normal in $\Smu(p)$. \end{proof}

Finally, we have:

\begin{claim} For $p\in S_G(M)$ and
  $|M|^+$-saturated extension $\CN\sube \CM$ we have
\begin{enumerate} \item  $\Smu(p)=\st^*(p(N){p(N)}^{-1})$. \item For
every $\alpha\models p$ in $N$, $\Smu(p)=\st^*(p(N)\alpha^{-1})$.\end{enumerate}
\end{claim}
\begin{proof} (1) Let us see that $\st^*\bigl(p(N){p(N)}^{-1}\bigr)\sub \Smu(p)$. Indeed,
assume that $g=\varepsilon \ccdot a\ccdot b^{-1}\in G(M)$ for $\varepsilon\in
\mu(N)$ and $a,b\in p(N)$. Then clearly $g\ccdot p$ is consistent with $\mu\ccdot
p$, so, since $g\ccdot p$ is a complete type, it follows that $g\ccdot p \vdash
\mu\ccdot p$. By Claim~\ref{claim:inf-eq}, $g\mu\ccdot p=\mu\ccdot p$ so $g\in
\Smu(p)$.

 For the opposite inclusion, assume that  $g\in \Smu(p)$ and take an arbitrary $b\models p$. Since $g\ccdot
p\vdash \mu\ccdot p$, there exist $\varepsilon\in \mu(N)$ and $a\in p(N)$
such that $ \varepsilon\ccdot a= g\ccdot b$ (here we used the saturation of $N$). It
follows that $g=\varepsilon a b^{-1}$, so $g\in \st^*(p(N){p(N)}^{-1})$.

For (2), use the fact that if $g=\st^*(\beta'\alpha'^{-1})$ is in $M$ and $\beta',
\alpha'\models p$, then by saturation of $\CN$ using an automorphism over $M$ we can replace $\alpha'$ with
$\alpha$ and $\beta'$ with some other $\beta\models p$.
\end{proof}

 This ends our discussion under the assumption that $G$ is
a general definably topological group.

\section{The case of an o-minimal $G$}

We assume in this section that $G$ is a definable group in an o-minimal structure
$\CM$ expanding a real closed field $R$.

\medskip

By Pillay's work (\cite{p}) we know that for any $k$, $G$  has a structure of a
definable $C^k$-manifold with respect to  $R$, making $G$ into a $C^k$-group. In
particular it is a definably topological group in $\CM$. Because $G$ has DCC (see
\cite{p}), it follows from Proposition~\ref{DCC} and Proposition~\ref{claim:stab2}
that for a definable type $p\in S_G(M)$, the groups $\Smu(p)$ and $\Stab(p)$ are
definable. Our goal in this section is to realize $\Smu(p)$ as the image under the
standard part map of a definable set in a bigger model. This will allow us to prove
for example that $\Smu(p)$ is a torsion-free group which is nontrivial unless $p$ is
bounded. It will also help us determining the dimension of $\Smu(p)$.

Most of the work in this section concerns the action of $G$ on itself by left
multiplication, but towards the end we also consider general definable $G$-sets in
the o-minimal setting.

\subsection{Embedding $G$ as an affine group}
The following claim is known, but since we could not find a precise
reference we provide an outline of a proof.
\begin{claim} The group $G$, with its $C^k$-manifold structure,  is definably
$C^k$-diffeomorphic to a closed $C^k$-submanifold of $M^n$ for some $n$.
\end{claim}
\begin{proof} This is done in two steps. First, we apply a result of Fischer (see
\cite[1.3]{F}), to find a $C^k$-diffeomorphism between $G$, and a $C^k$ submanifold
of $M^n$, for some $n$. Since Fischer states his result only for structures over the
reals, we clarify several points: The main tool in his article, Corollary 1.2, which
states that every closed set is the zero set of a definable $C^k$-function, can be
replaced by \cite[4.22]{DM}, whose proof goes through word-for-word for a general
real closed field. The rest of the argument (we only need Step 1 in that proof)
can remain as it is. We now identify  $G$ with its image in $M^n$.

Next we use the argument from \cite[6.18]{Coste}: as above we find a definable
$C^k$-function $h\colon M^n\to M$ whose zero locus is the closed set $Cl(G)\setminus
G$ and identify $G$ with the closed set $\{(x,t)\in M^{n+1}\colon  x\in G\, , \, t\ccdot
h(x)=1\}$.\end{proof}

We assume from now on that $G$ is a closed $C^1$-submanifold of $M^n$, with the group
operations $C^1$-maps.

\begin{rem} Note that now $p$ is an unbounded $G$-type if and only if
$p(\mathbb U)$ is not contained in any bounded subset of $\mathbb U^n$.
\end{rem}

\subsection{On the standard part map}
\label{sec:standard-part-map}

Recall that an elementary extension $\CN\sube \CM$ is called {\em
  tame\/}  if for every $n\in N^k$ the type
$\tp(n/M)$ is definable.

Let  $p$ be a type over $M$  and $\alpha$ be a realization of $p$.  By definability
of Skolem functions (see \cite{vdd}) the definable closure of $M\cup \{\alpha\}$ is
an  elementary extension of $\CM$  that we  will denote by $\CM\la \alpha\ra$. Clearly, when $p$ is a definable type, the extension    $\CM\la \alpha\ra\succeq
\CM$ is tame.
\medskip

Let $\CN\sube \CM$ be a tame extension and  $\CO=\CO_M(N)$  the set of $M$-bounded
elements of $N$, i.e.
\[ \CO=\{ n\in N \colon -m< n < m \text{ for some } m\in M\}.\]

We denote by $\nu=\nu_M(N)$ the set of $M$-infinitesimal elements of $N$, i.e.
\[ \nu=\{ n\in N \colon -m<n<m \text{ for any } m>0\in M \}. \] Note that $\nu^n$ is the
intersection of all $M$-definable open neighborhoods of $0\in M^n$.

Because $\CN$ is a tame extension of $\CM$ every element in $\CO$ is infinitesimally
close to a (unique)  element in $M$ (see \cite[Theorem 2.1]{MS}), hence  we have the
standard part map $\st\colon \CO\to M$ defined as: $\st(n)$ is the unique $m\in M$ such
that $n\in m+\nu$. We extend this definition to $\st\colon \CO^n\to M^n$. Note that in
this case the map $\st$ is the same as  our previous $\st^*$ with respect to the group
$G=\la M,+\ra$ or its cartesian powers.
\medskip

The group  $G$ is embedded as  a closed $C^k$-subamanifold of $M^n$. As before the
infinitesimal subgroup $\mu$ of $G$ is given by a definable basis  of $G$-open
neighborhoods  of $e\in G$. By the fact that the group topology agrees with the
restriction of the $M^n$-topology, we have $\mu(N)=(e+\nu^n)\cap G(N)$. Thus a
partial type which defines $\mu$ can be taken to be $\{B_\epsilon\colon \epsilon>0,
\epsilon\in M\}$, where we write $B_{\epsilon}$ for the intersection of the
$\epsilon$-ball in $M^n$ around $e$ with the  set $G$.

 Using continuity of
the group operations it is not difficult to show that for any $g\in G(M)$ we have
\[ g\ccdot\mu(N) = (g+\nu^n)\cap G(N). \]

Since $G$ is a closed subset, it follows that for any  $a\in \mathcal O(N)^n\cap
G(N)$ we have $\st(a)\in G(M)$ so we have $\mathcal O_G(N)=\mathcal O(N)^n\cap
G(N)$. In particular, the standard part map with respect to $G$ and  $\la M^n,+\ra$
coincide for elements of $G(N)$.

\subsection{Reduced types}

\label{sec:redef-infin-stab}

\begin{defn}
  We say that a type $p\in S_G(M)$ is {\em $\mu$-reduced\/}  if for every
  $q\in S_G(M)$ with $q_\mu=p_\mu$ we have $\dim (p)\leq \dim (q)$.
\end{defn}

By finiteness of dimension, for every $p\in S_G(M)$ we can find a $\mu$-reduced
$q\in S_G(M)$ with $p_\mu=q_\mu$. Note however that one $\sim_{\mu}$-class can
contain more than one $\mu$-reduced  types. E.g.\  in $\la M^2,+\ra$ the type of
$(\alpha,0)$ where $\alpha>M$ and the type of $(\alpha,1/\alpha)$ are both
$\mu$-reduced, one dimensional and $\sim_{\mu}$-equivalent.

\medskip

Our first goal is to show that for \emph{definable} $p\in S_G(M)$ we can find a \emph{definable } $\mu$-reduced $q$ with $p_\mu=q_\mu$.

\begin{claim}\label{claim:def-red}
Let $p\in S_G(M)$ be a definable type. If $p$ is not $\mu$-reduced  then there is a
{\em definable\/} type $q\in S_G(M)$ with $q_\mu=p_\mu$ and $\dim(q)< \dim(p)$.
\end{claim}

\begin{proof}
  Assume $p$ is a definable type that is not $\mu$-reduced.  Then we can
  find a type $s(x)\in S_G(M)$ such that $s_\mu=p_\mu$ and
$\dim(s)<\dim(p)$.  We can choose an $M$-definable set $V\subseteq G$ such that
$\dim(V)=\dim(s)$ and $s\vdash V$.

We fix some positive $r\in M$.

Since $p_\mu=s_\mu$ we have
\begin{equation}
  \label{eq:1}
p(x)\vdash \exists y \exists z ( x=y \ccdot z \,\&\, y\in B_{r} \,\&\, z\in V).
\end{equation}

We choose a realization $\alpha$ of $p$ and let $\CN=\CM\la \alpha \ra$.  Since the
type $p$ is definable, $\CN$ is a tame extension of $\CM$.

Using \eqref{eq:1} we obtain that there are $\beta\in V(N)$ and $g^*\in B_r(N)$ such
that $\alpha=g^*\ccdot \beta$.

Since $g^*\in B_{r}(N)$, we have $g^*\in \CO^n$. Let $g=\st(g^*)\in M$. Then
$g^*=g\ccdot \varepsilon$ for some $\varepsilon\in \mu(N)$. We have $\alpha=g\ccdot
(\varepsilon\ccdot \beta)$.

Let $q'(x)=\tp(\beta/M)$.  It is a definable type with $\dim(q')\leq \dim(V)<
\dim(p)$.

By \ref{claim:g-action} we also have $p_\mu= g\ccdot q'_\mu$. Now for the type
$q(x)=g^{-1}\ccdot q'$ we have that it is definable of the same dimension as $q'$ and
with $q_\mu=p_\mu$.
\end{proof}

\begin{cor}
\label{cor:def-reduced} For any definable type $p\in S_G(M)$ there is a
$\mu$-reduced
  definable type $q\in S_G(M)$ with $p_\mu=q_\mu$.
  \end{cor}

Finally, we easily have:
\begin{claim}\label{claim:simple} If $p\in S_G(M)$ is a $\mu$-reduced  type then for every $g\in G(M)$,
the type $g\ccdot p$ is $\mu$-reduced.
\end{claim}

We end with the following observation:
\begin{fact}\label{fact:algebraic-types} For a definable type $p\in
  S_G(M)$,  $p(\mathbb U)$ is bounded if and only if $p$ is $\mu$-equivalent to
 an algebraic  type $\tp(g/M)$, for some  $g\in
G(M)$.
\end{fact}
\begin{proof}
If $p\sim_\mu \tp(g/M)$ then any realization of $p$ is infinitesimally close to $g$.
But then, any $M$-definable open set containing $g$ must be in $p$ so $p$ has
formulas defining bounded sets. For the converse, if $p$ contains a formula over $M$
defining a bounded set $D$ then  $D(N)\sub \mathcal O_G(N)$, where $\CN=\CM\la
\alpha\ra$ for some $\alpha\models p$. It  follows that the standard part map is
defined on $D(N)$ and in particular, there exists $g\in G(M)$  infinitesimally close
to $\alpha$.\end{proof}

\subsection{Re-defining $\Smu(p)$ using the standard part map}

Our main goal in this section is to show that for a definable
$\mu$-reduced type $p$ we can
find an $M$-definable set $S$ in $p$ such that for any realization $\alpha\models p$ we
have $\Smu(p)=\st(S\alpha^{-1})$. We first clarify the notations.

Since $p$ is a definable type, the structure $\CN=\CM\la \alpha\ra$ is tame and we work
in $\CN$.  As before $\CO$ is the convex hull of $M$ in $\CN$ and we
have the standard part
map $\st\colon \CO^n \to M^n$. By \cite[Corollary 1.3]{vdd1}, for
every set $D$ definable in $\CN$ the  image $\st(D\cap \CO^n)$ is an $M$-definable set. We are going to omit
$\CO^n$ and just write $\st(D)$ in this case.  We still let $B_r$ denote the intersection
of the ball of radius $r$ centered at $e$ in $M^n$ with $G$.

\medskip
The first inclusion that we want is not difficult.

\begin{claim}\label{claim:stub-subset} 
Let $p\in S_G(M)$ be a definable type and
  $S$  an $M$-definable set in $p$. Then for every realization
  $\alpha\models p$  we have
\[ \Smu(p) \subseteq \st(S\alpha^{-1}), \]
where $\st$ is taken in the structure $\CN=\CM\la \alpha\ra$.
\end{claim}
\begin{proof} Let $\alpha$ be a realization of $p$. 

Assume $g\in \Smu(p)$. Then $g\ccdot p \vdash \mu\ccdot p$. 

As in Claim~\ref{claim:def-red} for every positive $r\in M$  we have
\[   g\ccdot p(x)\vdash \exists y \exists z ( x=y \ccdot z \,\&\, y\in B_{r} \,\&\,
z\in S). \] Thus for every positive $r\in M$ we have
\[ \CN \models \exists y \exists z (g\alpha= y\ccdot z \,\&\, y\in B_{r}\,\& \, z\in S). \]

In the structure $\CN$ we can now take the infinmum of all $r>0$ which satisfy the above.
This infimum belongs to $\nu(N)$ hence we can find $\varepsilon\in \nu(N)$ such that
\[ \CN \models \exists y \exists z (g\alpha= y\ccdot z \,\&\, y\in B_{\varepsilon}\,\& \, z\in S). \]

Hence there is $\beta\in S(N)$ (for $z$) and $g^*\in \mu(N)$ (for $y$) such that $g\ccdot
\alpha=g^*\ccdot \beta$.    Therefore $g=\st(\beta\ccdot \alpha^{-1})$.
\end{proof}

Recall (\cite[Proposition 1.10]{vdd1})  that for every definable set $V$ in an elementary
tame extension $\CN$ of $\CM$, $\dim (\st(V))\leq \dim V$. Since $\Smu(p)\sub \st(S\alpha^{-1})$
and we can choose $S$ with $\dim(p)=\dim(S)$, we have:

\begin{cor} For every definable type $p\in S_G(M)$, $\dim \Smu(p)\leq
\dim(p)$.\end{cor}

We can now state our main theorem.
\begin{thm}\label{Prop:definable}
Let $p\in S_G(M)$ be a definable $\mu$-reduced type. Then,
\begin{enumerate}[leftmargin=*]
\item $\dim \Smu(p)=\dim p$.

\item There exists an $M$-definable set $S$, with $p\vdash S$, such that $\dim S=\dim p$,
and for every $\alpha\models p$, and $\CN=\CM\la \alpha\ra$, we have

\noindent (i) $\Smu(p)=\st(S\alpha^{-1})$.

 \noindent (ii) The tangent space to $\Smu(p)$ at $e$ equals the standard part
   of the tangent space to $S\alpha^{-1}$ at $e$, i.e.
${T(\Smu(p))}_e = \st({T(S\alpha^{-1})}_e)$.
 \end{enumerate}
In particular, if $p$ is not a bounded type in $M^n$ then $\dim \Smu(p)>0$.
\end{thm}

\begin{rem}
In the case when $\dim p=1 $, clauses 1 and 2(i) of the above follow from \cite{PS},
and 2(ii) is contained in \cite{gp}.
\end{rem}

The proof of the above theorem will go through several steps and lemmas, and we divide it
into several subsections. The main point of the proof is to find an appropriate set $S$.

\subsubsection{\bf The existence of $S$ and the proof of of clause 2(i) in Theorem~\ref{Prop:definable}.}
\label{sec:proof-1}

\emph{During the rest of the proof we fix  a definable $\mu$-reduced
  type $p\in S_G(M)$, a realization
$\alpha\models p$ and $\CN=\CM\la \alpha\ra$. We work in the structure $\CN$.}

Notice that since $p$ is $\mu$-reduced,  for every $g\in G(M)$, the type $g\ccdot p$ is
also $\mu$-reduced. To simplify notation we use  from now on $\mathcal O$ for $\mathcal
O_G(N)$.

We first note:

\begin{claim}\label{claim:smallerdim} For every $M$-definable set $Y\sub G(N)$, if $\dim Y<\dim p$ then $\CO\ccdot
\alpha \cap Y=\emptyset$.
\end{claim}
\begin{proof} Assume for contradiction that $\CO\ccdot \al \cap Y\neq\emptyset$ and let
$\beta\in G(N)$ be a point in $Y\cap \mathcal O \ccdot \al$. If we let
$g=\st(\beta\alpha^{-1})$ then $tp(\beta/M)\sim_\mu g\ccdot p$ while
$\dim(tp(\beta/M))\leq \dim(Y)<\dim g\ccdot p$. This contradicts our above observation
that $g\ccdot p$ is $\mu$-reduced.
\end{proof}

\begin{claim}\label{claim:components2}
There exists an $M$-definable set $S$ in $p$ such that  every element of $S\cap (\mathcal
O\ccdot \alpha)$ realizes $p$.
\end{claim}

\begin{proof}
For every definable set $S\in p$, the set $S\alpha^{-1}\cap \mathcal O$ is a relatively
definable subset of
 $\CO\sub  {\CO_M(N)}^n$, hence, by Theorem~\ref{thm:conv-def}
 in Appendix~\ref{sec:setting} (see also
Example~\ref{sample:conv-def}) it has finitely  many connected components. Namely, it
can be written as a  finite union of   pairwise disjoint, relatively definable subsets of
$\CO$, each of which is clopen in $S\al^{-1}\cap \CO$ and such that any other
relatively definable clopen subset of $S\al^{-1}\cap \CO$ contains one of those.

We choose an  $M$-definable $S$ in $p$  such that $\dim S=\dim p$, $S$ is a cell, and the
number of connected components of $S\alpha^{-1}\cap \mathcal O$ is minimal. We claim that
$S$ has the desired property.

Indeed, assume not, namely there exists $\beta\in S\cap (\mathcal O\ccdot \alpha)$ such
that $\beta\models q\in S_G(M)$ and $q\neq p$. By Claim~\ref{claim:smallerdim}, we must
have $\dim q=\dim p=\dim S$. Since $p\neq q$ there exists an $M$-definable set $Y$ in $q$
but not in $p$. We may assume that $Y\sub S$ and furthermore that $Y$ is relatively open
in $S$ (because $\dim q=\dim S$), so $Y\alpha^{-1}$ is relatively open in $S\alpha^{-1}$.

We claim that  $Y\alpha^{-1}\cap \mathcal O$ cannot be relatively closed in
$S\alpha^{-1}\cap \mathcal O$. Indeed, if it were closed then $Y\alpha^{-1}\cap \mathcal O$
would be clopen in $S\al^{-1} \cap \mathcal O$ and therefore would contain the whole connected
component of $\beta$ in $S\alpha^{-1}\cap \mathcal O$. This would imply that the number
of components of  $(S\setminus Y)\al^{-1}\cap \CO$ is smaller than that of $S\al^{-1}\cap
\CO$ and furthermore $\al\in S\setminus Y$. Since $S\setminus Y$ is defined over $M$ we
obtain a contradiction to the minimality of components in $S\al^{-1}\cap \CO$.

Thus, $Y\cap \CO\ccdot \al$ is not closed in $S\cap \CO\ccdot \al $, so the set
$Fr(Y)\cap \mathcal O \ccdot \alpha$ is non empty. However, $Fr(Y)$ is $M$-definable and
its dimension is smaller than that of $Y$, so also smaller than $\dim p$. This
contradicts Claim~\ref{claim:smallerdim}, so ends the proof of
Claim~\ref{claim:components2}.
\end{proof}

In the rest we will show that any set $S$ as in Claim~\ref{claim:components2} satisfies the clause 2 of Theorem~\ref{Prop:definable}.

\begin{claim}\label{claim:one} Let $S$ be as in Claim~\ref{claim:components2} and assume that $\alpha\models p$.
Then  $\Smu(p)=\st(S\alpha^{-1})$. Moreover, for every $M$-definable set $S_1\sub S$, if
$p\vdash S_1$ then $\Smu(p)=\st(S_1\alpha^{-1})$ \end{claim}

 \begin{proof} By Claim~\ref{claim:stub-subset}, we have $\Smu(p)\subseteq \st(S_1\alpha^{-1})$, so if we prove
that $\st(S\alpha^{-1})\sub \Smu(p)$ then we
get equality.

Assume that $g\in \st(S\alpha^{-1})$. Then there exists $\epsilon \in \mu(N)$ and
$\beta\in S$ such that $\epsilon g=\beta\alpha^{-1}$, so $ \epsilon g
\alpha=\beta$. 
By the choice of $S$ we have $\beta\models p$ which implies
that the types $g \ccdot p(x)$ and  $\mu\ccdot p(x)$ are mutually consistent so $g\in
\Smu(p)$. Thus, $\st(S\alpha^{-1})\subseteq \Smu(p)$.

For the moreover part, it is easy to see that any such $S_1$ also
satisfies Claim~\ref{claim:components2}.
\end{proof}

This ends the proof of clause 2(i) in Theorem~\ref{Prop:definable}.

\subsubsection{\bf The dimension of $\Smu(p)$ and the proof of clause 1
in Theorem~\ref{Prop:definable}.}
\label{sec:proof-2}

 We start by proving, via  sequence of claims,  a general proposition. Below we
write $\dim_\CM$ and $\dim_\CN$ to emphasize that we compute the dimension of
$\CM$-definable and $\CN$-definable sets, respectively. Recall that we are using $\nu(N)$
to denote the infinitesimals of $\CM$ in $\CN$.

\begin{prop}\label{claim:general} Let $U\subseteq N^k$ be an
  $M$-definable open set containing $0$,  and  $Y\sub N^k$ be  an $N$-definable,  relatively closed
$C^1$-submanifold of  $U$ with $0\in Y$.  Assume that for every $h_1,h_2\in {\nu(N)}^k\cap
Y$ we have $\st({T(Y)}_{h_1})=\st({T(Y)}_{h_2})$. Then $\dim_\CM(\st(Y))=\dim_\CN(Y)$.
\end{prop}
\begin{proof} Assume that $\dim(Y)=\ell$. By \cite{vdd1}, $\dim_\CM(\st(Y))\leq \dim_\CN(Y)$ so
it is sufficient to prove that $\dim_\CM(\st(Y))\geq \ell$.

Let $H_0=\st({T(Y)}_0)$. So $H_0$ is a  linear subspace of $M^k$ of dimension
$\ell$.  Doing a linear change of variables defined over $M$ we may assume that
$H_0=M^{\ell}$, identified with the first $\ell$ coordinates of $M^k$. Let $\ell'=k-\ell$
and write $N^k=N^\ell\times N^{\ell'}$. We will denote by $\pi\colon N^k\to N^\ell$ the
projection onto the first $\ell$ coordinates.

The following is a special case of the implicit function theorem.

\begin{fact}[Implicit Function Theorem]\label{fact:imp} 
Let $Z\subseteq N^k$ be a definable $C^1$-submanifold
  of dimension $\ell$, $a\in Z$, $b=\pi(a)$, and $L={T(Z)}_a$.
Assume $\pi(L)=N^\ell$.  Then locally, near $a$, $Z$ is the graph of a function
$F\colon N^\ell\to N^{\ell'}$, and $L$ is the graph of the differential of $F$ at $b$.
\end{fact}

\begin{claim}\label{claim:inf-diff} 
If $h\in Y\cap {\nu(N)}^k$ then, near $h$, $Y$ is
  the graph of a function $F\colon N^\ell\to N^{\ell'}$. Moreover $\|
  dF_b\|\in \nu(N)$, where $b=\pi(h)$.
\end{claim}
\begin{proof}
  This follows from our assumption on ${T(Y)}_h$ and Fact~\ref{fact:imp}.
\end{proof}

The following is easy to prove.

\begin{claim}\label{claim:gen-mu}
Let $D\subseteq {\nu(N)}^\ell$ be an open ball centered at  $0$ and $F\colon D\to
N^{\ell'}$ a definable smooth function. Assume $F(0)=0$ and $\|dF_b\|\in \nu(N)$ for all
$b\in D$. Then $F(D)\subseteq {\nu(N)}^{\ell'}$ and therefore the graph of $F$ is contained
in ${\nu(N)}^k$.
\end{claim}

Combining   Claim~\ref{claim:inf-diff}, Claim~\ref{claim:gen-mu} with the fact that
$Y\subseteq   U$ is a closed submanifold, we obtain:
\begin{claim}
  \label{claim:contains-graph-ball}
Let $a>0$ be in  $\nu(N)$ and $D_a$ be the open ball of radius $a$ in $N^\ell$ centered
at $0$.  Let $Y_0^a$ be the connected component of  $Y\cap (D_a\times N^{\ell'})$
containing $0$.
 Then $Y_0^a \subseteq {\nu(N)}^k$ and it is the graph of a definable function
$F\colon D_a\to N^{\ell'}$.
\end{claim}

By o-minimality we conclude:
\begin{claim}
There is $a>\nu(N)$ in $N$ such that $Y$ contains the graph of a definable function
$F\colon B_a\to N^{\ell'}$, where $B_a$ is the open ball of radius $a$
in $N^k$ centered at $0$. In particular, $\pi(Y)$ contains an open ball $B_b \subseteq
N^\ell$ centered at $0$ or radius $b$ with $b\in M$.
\end{claim}

We can now complete the proof of Proposition~\ref{claim:general}. Because $\pi(Y)$
contains an $M$-definable ball around $0$, the set  $\st(\pi(Y))$ contains an
$M$-definable neighborhood of $0$ in $M^\ell$, so in particular its
dimension is at least $\ell$.
But $\st(\pi(Y))=\pi(\st(Y))$, hence $\dim_\CM(\st(Y))\geq \ell$.
\end{proof}

We now return to the proof of clause 1 in Theorem~\ref{Prop:definable}. As before, we assume that $G\sub M^n$ is an
embedded $k$-dimensional closed $C^1$-submanifold.

\begin{claim} \label{claim:1}
 Let $Y\sub G$ be an 
$M$-definable set with $p\vdash Y$ and $\dim Y=\dim p=\ell$. Then for every $m\in M$,
$B_m\ccdot \alpha\cap Y$ is a closed $\ell$-dimensional submanifold of $B_m\ccdot
\alpha$.
\end{claim}

\begin{proof} By o-minimality, the set $Z$ of all points in $Y$ at
  which $Y$ is {\bf not} an $\ell$-dimensional submanifold of $G$
is an $M$-definable subset of smaller dimension. By Claim~\ref{claim:smallerdim},
$(B_m\ccdot \alpha)\cap Z=\emptyset$. Similarly, the intersection of $Fr(Y)$ with
$B_m\ccdot \al$ is empty, so $B_m\ccdot \alpha\cap Y$ is a closed submanifold.
\end{proof}

 By working in a ($M$-definable) chart
of $G$ near $e$ we  identify $G$ locally at $e$ with an open
neighborhood $U$ of $0$ in $M^k$.  Since $U$ is an open
subset of $M^k$, for each $g\in U$ the tangent space ${T(U)}_g$ can
be identified with $M^k$.  Thus working in $\CN$, for $g\in G(N)\cap
U(N)$ we identify the tangent space ${T(G)}_g$ with $N^k$.

For $g\in G$, let $r_g\colon G\to G$ be the right multiplication by
$g$.   The differential  ${d(r_g)}_e$ of $r_g$ at $e$, is a linear
isomorphism from ${T(G)}_e$ to ${T(G)}_g$,  and for $g\in G(N)\cap U(N)$ we
will view  ${d(r_g)}_e$ as a linear isomorphism from $N^k$ to
$N^k$, i.e.\  an element of $\GL_k(N)$.  Thus we
have a continuous, $M$-definable map  $g\mapsto {d(r_g)}_e$ from $U(N)$ to $\GL_k(N)$.

Let $S$ be an $M$-definable set as in Claim~\ref{claim:components2}.  
Then   every $\beta \in \CO\ccdot  \alpha\cap S$
realizes $p$ and, by Claim~\ref{claim:one},
$\Smu(p)=\st(S\al^{-1})$.  Replacing $U$ with $B_1\cap U$ if needed, 
and using  Claim~\ref{claim:1}, we may assume that $S\alpha^{-1}\cap
U(N)$ is a closed submanifold of $U(N)$. 

\begin{claim}\label{label:claim1.7} For every $h\in \mu(N)\cap S\alpha^{-1}$,
$$\st({T(S\alpha^{-1})}_h)=\st({T(S\alpha^{-1})}_e).$$
\end{claim}
\begin{proof} If $h\in \mu(N)\cap S\alpha^{-1}$ then it is of the form $h=\beta\al^{-1}$ for some $\beta
\models p$. But then, since $\alpha\equiv_M\beta$,
$\st({T(S\beta^{-1})}_e)=\st({T(S\al^{-1})}_e)$. The map $r_{h}$ sends
$S\beta^{-1}$ to $S\al^{-1}$ with $e$ going to $h$, so its differential at $e$ sends
${T(S\beta^{-1})}_e$ to ${T(S\al^{-1})}_h$. Since $h\in \mu(N)$ and the
map $h\mapsto {d(r_h)}_e$ is continuous and $M$-definable,  viewing ${d(r_h)}_e$ as an
element of $\GL_k(N)$, we can write it as ${d(r_h)}_e= I +\varepsilon$,
where $\varepsilon$ is a $k\times k$ matrix whose entries are in
$\nu(N)$.  Applying ${d(r_h)}_e$  to an orthonormal basis of
${T(S\beta^{-1})}_e$ (with respect to the standard dot product in $N^k$)
we conclude that for each such basis vector
$v$, $\st({d(r_h)}_e(v))=\st(v)$. It follows that $\st({T(S\beta^{-1})}_e)=\st({T(S\al^{-1})}_h)$,
so we get the desired result.\end{proof}

Using Claim~\ref{label:claim1.7}, and Proposition~\ref{claim:general}  we conclude that
$\dim \Smu(p)=\dim S=\dim(p)$, namely we end the proof of clause 1 in Theorem~\ref{Prop:definable}.

\subsubsection{\bf Proof of clause 2(ii) in Theorem~\ref{Prop:definable}}
\label{sec:proof-3}
As in the previous section, we identify $G$ near $e$ with an open
$M$-definable subset $U\subseteq M^k$ containing $0$. 
Again, shrinking $U$ if needed 
we assume that $U$ is contained in  $B_1$.

Note that in general it is not true that ${T(\st(Y))}_a=\st({T(Y)}_{\st(a)})$, even for a
smooth definable manifold $Y$. However, we use the next Fact, which easily
follows from Marikova's result (\cite[Theorem 2.23]{Marikova}):
\begin{fact} Let $Y\sub {\CO(N)}^k$
be an $N$-definable submanifold of dimension $\ell$. Assume also that $\dim_\CM
\st(Y)=\dim_\CN Y$. Then there exists $y\in Y$ such that
$\st(T(Y)_y)=T(\st(Y))_{\st(y)}$.\end{fact}

Let $S$ be an $M$-definable set as in Claim~\ref{claim:components2} and $Y=S\alpha^{-1}\cap U(N)$.  We need to show
that $T(\Smu(p))_e = \st(T(Y)_e)$.

We apply the above fact to $Y$, and fix $h\in Y$ with
$\st(T(Y)_h)=T(\st(Y)_{\st(h)})$.

As in the proof of Claim~\ref{label:claim1.7}, $h=\beta\alpha^{-1}$ with $\beta\models p$, and
$$d(r_h)_e(T(S\beta^{-1})_e)=T(Y)_h.$$

 Let $g=\st(h)$. By continuity of the map $h\mapsto d(r_h)_e$, as a map from $U(N)$ to $\GL_k(N)$, we have
$$d(r_{g})_e(\st(T(S\beta^{-1})_e))=\st(T(Y)_h).$$ However,
since $\beta\equiv_M \alpha$,  $\st(T(S\beta^{-1})_e)=\st(T(Y)_e)$. We
conclude that $$d(r_{g})_e(\st(T(Y)_e))=\st(T(Y)_h).$$ By our assumption on $h$,
we have $\st(T(Y)_h)=T(\st(Y))_g$, hence 
$$d(r_g)_e(\st(T(Y)_e))=T(\st(Y))_g.$$

By Claim~\ref{claim:one} we have that $\Smu(p)\cap U= \st(Y)\cap U$, hence
$T(\Smu(p))_g = T(\st(Y))_g$, and 
$$d(r_g)_e(\st(T(Y)_e))= T(\Smu(p))_g.
$$
Since $g\in \Smu(p)$ and $\Smu(p)$ is a subgroup of $G$, the map $r_g$ is a diffeomorphism  of $\Smu(p)$
to itself,  hence $ T(\Smu(p))_g = d(r_g)_e( T(\Smu(p))_e)$, and $d(r_g)_e(\st(T(Y)_e))= d(r_g)_e( T(\Smu(p))_e)$.

The linear map $d(r_g)_e$ is invertible, hence $\st(T(Y)_e)=  T(\Smu(p))_e$
which is what we wanted to prove.

This ends the proof of   Theorem~\ref{Prop:definable}.

\subsection{The structure of $\Smu(p)$}

 Below  $\CM$ is an o-minimal expansion of
a real closed field and $G$ a definable group.

We  begin with  an observation:
\begin{claim} \label{claim:type-in-group} If $p\in S_G(M)$ is a
  definable type and $H$ is an $M$-definable subgroup of $G$ with $p\vdash H$ 
   then $\Smu(p)\sub H$.
\end{claim}
\begin{proof}Let $S,\alpha$ and $\CN$ be as in Theorem~\ref{Prop:definable}. Replacing $S$ by $S\cap H$ if needed  we may
  assume  $S\sub H$.  Then $S\ccdot
\alpha^{-1}$ is also contained in $H(N)$. Since $H$ is closed in $G$
we get  $\st(S\ccdot \alpha^{-1})\sub H$.\end{proof}

We will need  the following fact about definable groups.

\begin{fact}\label{Borel}
Let $G$ be a definable, definably connected group in $\CM$. Then there exists an
$M$-definable solvable, torsion free subgroup $H_1\sub G$ and a definably compact
set $C\sub G$ such that $G=C\ccdot H_1$. In particular, $G/H_1$ is a definably
compact space.
\end{fact}
\begin{proof} We first prove the existence of a torsion free $H_1$ such that that $G/H_1$ is
definably compact. This basically follows from the work of A. Conversano, but we
give the details.

Use induction on $\dim G$. If $G$ is not semisimple then it has an infinite
$M$-definable normal abelian subgroup $N$. By induction the group $G/N$ has an
$M$-definable  solvable torsion free subgroup, which we may assume is of the form
$H/N$,  such that $(G/N)/(H/N)$ is a definably compact space. But then, the group
$H$ is clearly solvable as well, and the quotient $G/H$ is isomorphic to
$(G/N)/(H/N)$ so definably compact. By \cite[Proposition 2.2]{CP}, $H$ has a maximal
normal torsion-free definable subgroup $H_1\unlhd H$ with $H/H_1$ definably compact.
It follows that the space $G/H_1$ is definably compact as well.

Assume then that $G$ is semi-simple. Then by \cite[Theorem 1.2]{conversano1}, $G$
can be written as a product of two subgroups $G=K\ccdot H_1$ for $K$ a definably
compact and $H_1$ torsion-free  (so necessarily solvable). This clearly implies that
$G/H_1$ is definably compact.

Let us prove now the existence of a definably compact $C\sub G$ such that $G=C\cdot
H_1$. We first note that $G$ can be written as an increasing union of open sets
$G=\bigcup_{r>0} B_r$, such that for each $r$, $Cl(B_r)$ is definably compact (here
we use the fact that $G$ is embedded in $M^n$). We also have $G/H_1=\bigcup_r
\pi(B_r)$ and since $\pi$ is an open map each $\pi(B_r)$ is open. Since $G/H_1$ is
definably compact there is $r_0$ such that $G/H_1=\pi(B_{r_0})$, but then
$G=B_{r_0}\cdot H_1$ so in particular $G=Cl(B_{r_0})\ccdot H_1$.
 \end{proof}

\begin{thm}\label{theorem:1}
\begin{enumerate}[leftmargin=*]
\item Let $p$ be a definable type in $S_G(M)$. Then
the group $H=\Smu(p)$  is solvable, torsion-free, with
$\dim H\leq \dim p$. In particular, $\Stab(p)$ is torsion-free as well.

\item
 In the opposite direction, if $H$ is a torsion-free group definable over $M$ then there exists
 a complete definable $H$-type $p$ over $M$ such that
 $\Smu(p)=\Stab(p)=H$.
\item Every two maximal torsion free definable subgroups of $G$ are conjugate.
\end{enumerate}
\end{thm}

\begin{proof}  (1) Let $H_1\sub G$ be any $M$-definable torsion-free solvable group as in Fact~\ref{Borel} and let $C\sub G$ be a definably compact set, defined over $M$, such
that $C\ccdot H_1=G$.

We take $\alpha\models p$ and work inside $\CN=\CM\langle \alpha\rangle$. There is
some $g\in C(N)$ such that $g^{-1}\alpha=\beta \in H_1$. Because $C$ is an
$M$-definable,  definably compact set, there exists  some $g_0\in C(M)$ such that
$g_0g^{-1}\in \mu(N)$. It follows that $g_0^{-1}\alpha \in \mu(N)\ccdot  \beta$ and
since $g_0\in M$ we have $g_0^{-1}\alpha\models g_0^{-1}\ccdot p$. If we let
$q=tp(\beta/M)$ then this implies $(g_0^{-1}\ccdot p)_{\mu}=q_{\mu}$.

Because $\beta\in H_1$, it follows from Claim~\ref{claim:type-in-group} that
$\Smu(q)$ is a subgroup of $H_1$ and
 hence $\Smu(p)$ is conjugate, by an element of $G(M)$,
to a definable subgroup of $H_1$. Clearly, every  such group is solvable,
torsion-free.

(2) This is exactly Proposition 4.7 in \cite{CP}.

(3) Let $H_1\subseteq G$ be a maximal torsion free subgroup as in Fact\ref{Borel}.
Let $H\subseteq G$ be a definable torsion-free  subgroup. By (2) there is a complete
$H$-type $p$ with   $\Smu(p)=\Stab(p)=H$.  By the proof of (1), $H=\Smu(p)$ is
conjugate to a subgroup of $H_2<H_1$. As $H$ is maximal, we get  $H_2=H_1$.
\end{proof}

\subsection{Stabilizers of types in definable $G$-sets.}

All that we have done so far in the o-minimal setting was to analyze $S^\mu_G(M)$.
 We present here several consequences for  definable $G$-sets and leave the more
 substantial
 investigation for further research. We thus fix a definable group $G$ and a
 definable $G$-set $X$,  in an o-minimal expansion $\CM$
 of a real closed field.

We first observe the following:
\begin{prop}\label{no-torsion}
\begin{enumerate}
\item The group $\mu(\mathbb U)$ is torsion-free.

\item If $\CN$ is a tame extension of $\CM$ then for every $N$-definable subgroup
$H$, the $M$-definable group $\st(H)$ has the same dimension as $H$.
\end{enumerate}
\end{prop}
\begin{proof} (1) Assume for contradiction that $\mu(\mathbb U)$ contains an $n$-torsion
point of $G$. Thus the set $\{g\in G\colon  g\neq e \mbox{ and } g^n=e\}$ is an
$M$-definable set whose closure contains $e$.

We consider the Lie Algebra $\mathfrak{g}$, associated to $G$ (see \cite{PPS}). It
is not hard to see that for every $n\in \mathbb N$, the differential of the map
$g\mapsto g^n$ at $e$, call it $d_n\colon \mathfrak g\to \mathfrak g$, is just
the map $v\mapsto n v$. Since it is invertible, the map $g\mapsto g^n$ must be
injective in a small neighborhood of $e$, contradicting the fact that every
neighborhood of $e$ contains an $n$-torsion point.

(2) Assume that $H$ is an $N$-definable subgroup of $G$. As noted before, for every
$h\in H$, $d(r_h)_e(T(H)_e)=T(H)_h$. Hence for all
$h\in \mu(\mathbb U)\cap H$, $\st(T(H)_h)=\st(T(H)_e)$. But then, by Proposition~\ref{claim:general}, $\dim (\st(H))=\dim H$.
\end{proof}

Regarding the above proposition, note that while $\dim(\st(H))=\dim(H)$, these
groups could be quite different. For example, let $H$ be the $\SO(2,R)^A$, where
$$A=\begin{pmatrix}
  \al & 0 \\
  0 & \al^{-1} \\
\end{pmatrix},$$ for $\al$ realizing the type of an infinitely large element.
The group $H$ is definably compact but
$$\st(H)=\left\{\begin{pmatrix}
  1 & a \\
  0 & 1 \\
\end{pmatrix}: a\in \mathbb R\right\}$$ is torsion-free.

 We can now state two results on $\Smu(p)$ for definable $G$-sets.
\begin{prop} Assume that $X$ is a definable $G$-set, $p\in S_X(M)$
a definable type and $\al\models p$ is in $\mathbb U$.
\begin{enumerate}
\item Let $G_\al=\{g\in G(\mathbb U)\colon g\ccdot \al=\al\}$. Then
$\dim(\Smu(p))\geq\dim (G_\al)$.

\item Assume that $G$ acts transitively on $X$, and endow $X$ with the induced
topology (either through $S^\mu_X(M)$, see Corollary~\ref{cor:G-mod-closed}, or by
identifying $X$ with $G/G_a$ for some $a\in X(M)$). If $p$ is unbounded with respect
to this topology then $\dim (\Smu(p))>0$.
\end{enumerate}
\end{prop}
\begin{proof} (1) We work in $\CN=\CM\la \al \ra$ a tame extension of $\CM$. By Claim~\ref{normal}, the group $\st(G_\al)$ is a subgroup of $\Smu(p)$ definable in $\CM$,
and by Proposition~\ref{no-torsion}, its dimension equals to that of $G_\al$. Hence
$\dim(\Smu(p))\geq\dim (G_\al)$.

(2) Fix $a\in X(M)$. By definable choice we can
find an $M$-definable set $Y\sub G$ which is in definable bijection with $X$ via the
map $\pi(g)=g\ccdot a$. While $\pi$ is not a homeomorphism of $Y\sub G$ and $X$
(with its quotient topology), if $D\sub Y$ is a definable set such that $Cl(D)$ is
definably compact in $G$, then $\pi(Cl(D))$ is definably compact with respect to the
topology of $X$.

Assume now that  $p\in S_X(M)$ is a definable type which is not bounded, namely,
does not contain any definably compact $X$-formula. Let $q\in S_Y(M)$ be the pull-back
of $p$ under $\pi$ (so $q\ccdot a=p$). Then $q$ is a definable $G$-type which, by
the above discussion, is unbounded in $G$. By Theorem~\ref{Prop:definable} the
$\mu$-stabilizer of $q$ has positive dimension and it is easy to see that
$\Smu(q)\sub \Smu(p)$. We thus showed that $\Smu(p)$ is infinite.
\end{proof}

\section{Some examples}

\subsection{The group $G=\la \mathbb R^2,+\ra$}

Let $\CM$ be an o-minimal expansion of the real field, and $G=\la \mathbb R^2,+\ra$. Our
goal is to understand the space $S_G(M)$ and the $\mu$-stabilizers of types. Note that
every type in $S_G(M)$ is definable (see \cite{MS}).  In our discussion below, $p$ is a
complete type in $S_G(M)$.

\medskip

 \noindent{\bf Bounded
types.}  As we pointed earlier, if $p$ contains any formula which defines a bounded
subset of $\mathbb R^2$ then it is $\mu$-equivalent to a (type of an) element $g\in
\mathbb R^2$ and hence $\Smu(p)=\{e\}$ (here $e=(0,0)$).
\medskip

\noindent{\bf Unbounded types of dimension $1$ }. We may assume that there is a
definable unbounded curve $\gamma(t)=(\gamma_1(t),\gamma_2(t))\colon (0,\infty)\to \mathbb
R^2$ such that $p$ is the type of $\gamma$ ``at $\infty$''. Clearly, $p$ is
$\mu$-reduced , for otherwise it will be infinitesimally close to a point $g\in
\mathbb R^2$, contradicting the fact that it is unbounded.

If we let $$\vec{v}=\lim_{t\to \infty} \frac{\dot{\gamma(t)}}{||\dot{\gamma(t)}||}.$$
then $\Smu(p)=\mathbb R\vec{v}$.
\medskip

\noindent{\bf Unbounded types of dimension $2$}. We prove here a general claim:

\begin{claim} If $G$ is any definable group in an o-minimal structure $\CM$ and $p\in
S_G(M)$ is a definable $\mu$-reduced  type with $\dim(p)=\dim(G)$ then $\Stab(p)=\Smu(p)=G$.
\end{claim}
\begin{proof} By Claim~\ref{claim:components2}, there is an $M$
  definable set $S$  in $p$, such that for every $\alpha\models p$ and $m\in M$, every
element in $(B_m\ccdot \alpha)\cap S$ realizes $p$. Since $\dim S=\dim
p=\dim G$, by Claim~\ref{claim:smallerdim}, every element in $B_m\ccdot \alpha$ realizes $p$. Because $G(M)$ is the
union of these $B_m(M)$'s, it follows that every element in $G(M)\ccdot \alpha$
realizes $p$, so $G(M)=\Stab(p)$.\end{proof}

Going back to our example, we may conclude that $\Stab(p)=\Smu(p)=\mathbb R^2$.

\subsection{The action of $\SL(2,\mathbb R)$ on $\mathbb H$}

We now let $G=\SL(2,\mathbb R)$ and consider its usual action on the upper half plane
$\mathbb H=\{z\in \mathbb C\colon  Im(z)>0\}$ via
$$\begin{pmatrix}
  a & b \\
  c & d \\
\end{pmatrix}
z = \frac{az+b}{cz+d}.$$ The action is transitive and the
stabilizer of each point $z\in \mathbb H$, call it $G_z$, is a conjugate of
$\SO(2,\mathbb R)$.

Our goal is to understand $\mu$-types in $\mathbb H$. We work in an
o-minimal expansion $\CM$ of the field $\RR$.

\medskip

 We are using the fact that there is a definable compactification of $\mathbb H$, call
it $\bbar{ \mathbb H}$ such that the action of $G$ can be definably extended to
that of $\bH$. Namely,

$$\bH=\{z\in \mathbb C\colon Im(z)\geq 0\}\cup \{\infty\},$$
with the action of $G$ on the real line given by the same linear fraction, and with
$A\ccdot z=\infty$ when $cz+d=0$ and $A\ccdot \infty =a/c$. Clearly, the action is
transitive on $\bH\setminus \mathbb H$. The topology on $\bH$ is as follows: the
induced topology on $\bH\cap \mathbb C$ is the Euclidean one, and neighborhoods of
$\infty$ are of the form $g\ccdot U$, where $U\sub \bH$ is a neighborhood of $0$ (by
transitivity we could have chosen here any point other than $0$). The space $\bH$ is
compact and the action of $G$ is continuous. Note however, that because
$\bH\setminus \mathbb H$ is a closed  orbit of $G$, it is not true for $B\sub G$
open and $z\in\bH\setminus \mathbb H$ that the set $B\ccdot z$ is open.

\medskip

Given a complete type $p\in S_{\mathbb H}(M)$ we say that $z\in \bH$
is a limit point of $p$  if for every $M$-definable open neighborhood $U\sub \bH$
of $z$, the formula defining $U\cap \mathbb H$ is in $p$. Since $p$ is a complete
type and $\bH$ is a Hausdorff space each $p$ has at most one limit
point in $\bH$.

Consider now the intersection $\bigcap_{p\vdash Y} Cl(Y)$, where  $Cl(Y)$ denotes the closure of $Y$ in  $\bH$. Since $\bH$ is
compact and the collection of these $Y$'s is finitely satisfiable it follows that
the intersection is non-empty. It is not hard to see that any $z$ in this
intersection is a limit point of  $p$. Hence every $p\in S_{\mathbb H}(M)$ has a
unique limit point, call it $\lim p$, in $\bH$.

\begin{claim} The map $p\mapsto \lim (p)$ factors through $S^\mu_{\mathbb H}(M)$
and induces a topological homeomorphism of $G$-spaces $S^\mu_{\mathbb
H}(M)$ and $\bH$.
\end{claim}

\proof In order to see that the map factors through $S^\mu_{\mathbb H}(M)$ we need to
show: If $p,q\in S_{\mathbb H}(M)$ and $z_p=\lim p\neq \lim q=z_q$ then $p_\mu\neq
q_\mu$.

Without loss of generality, $z_p, z_q\in \mathbb C\cap \bH$. If either of them is in
$\mathbb H$ then it is easy to see that $p_\mu\neq q_\mu$ so we may assume that
$z_p,z_q\in \mathbb R$. Now, let $D_p,D_q\sub \mathbb C$ be open discs centered at
$z_p, z_q$, respectively, such that $D_p\cap D_q=\emptyset$. By continuity of the
action on $\bH$, there exist open discs $U_p\sub D_p$ around $z_p$ and $U_q\sub D_q$
around $z_q$ and a definable open neighborhood $W$ of $e$ in $G$  such that $W\ccdot
U_p\sub D_p$ and $W \ccdot U_q\sub D_q$. By definition of $\lim p$ and $\lim q$ it
follows that $p\vdash U_p$ and $q\vdash U_q$, hence $\mu\ccdot p\sub D_p$ and
$\mu\ccdot q\vdash U_q$, so $p_\mu\neq q_\mu$.

A small variation of the above argument  shows that the induced map
$\pi\colon S_{\mathbb
H}^\mu(M)\to \bH$ is continuous. Since $\mathbb H$ is dense in $\bH$ and $S^\mu_{\mathbb
H}(M)$ is compact, the map is surjective. Let us see that it is also injective.

Assume that $\lim p=\lim q$; we will show that $p_\mu=q_\mu$. Without loss of generality,
we may assume that $\lim p=\lim q=0$, and to simplify matters it will be sufficient to
take $q=tp(\beta i/\mathbb R)$, with $\beta>0$ in an elementary
extension of $\CM$ infinitesimally close to $0$. It is
clearly enough to show that $p_\mu=q_\mu$. Let $\alpha\models p$ in some elementary
extension. We will show that there exists $h \in \mu(G)$ such that $h\ccdot \alpha \models
q$.

Indeed, write $\alpha=\alpha_1+\alpha_2 i$ with $\alpha_1, \alpha_2$ infinitesimally
close to $0$ and take
$$h=\left ( \begin{array}{cc}
  1 & -\alpha_1 \\
  0 & 1 \\
\end{array}\right).$$ Clearly, $h\in\mu(G)$ and we have $h\ccdot\alpha=\alpha_2 i\models q$. It follows, using \ref{claim:inf-eq} that
$p_\mu=q_\mu$.

We therefore showed that $\pi\colon S^\mu_{\mathbb H}(M) \to \bH$ is a continuous bijection.
Since both are compact and Hausdorff it is actually a homeomorphism. \qed

As an immediate corollary we obtain:

\begin{claim} If $p_\mu \in S^\mu_{\mathbb H}(M)$  and $z_p=\lim p \in \bH$  then
$\Smu(p)=G_{z_p}$. Hence, if $\lim p\in \mathbb H$ then $\Smu(p)$ is a conjugate of
$\SO(2,\mathbb R)$ and if $\lim (p)\in \bH\setminus \mathbb H$ then $\Smu(p)$ is a
conjugate of the solvable group

$$\left\{\left ( \begin{array}{cc}
  a & b \\
  0 & 1/a \\
\end{array}\right )\colon a,b\in \mathbb R\right\}.$$

\end{claim}

\appendix
\renewcommand{\thesection}{Appendix~\Alph{section}}

\section{More on the space $S_X^\mu(M)$}

\renewcommand{\thesection}{\Alph{section}}

In Topological Dynamics the Samuel compactification $S(G)$ of a topological group
$G$  is the  compactification of $G$ with respect to the uniformity induced via the
action of $G$ on itself by left multiplication.  Up to a homeomorphism it is the
unique compact space with the following properties.
\begin{enumerate}[label=(\roman*), leftmargin=*]
\item There is an embedding $\rchi\colon G\to S(G)$ such that the image
  $\rchi(G)$ is dense in $S(G)$, and  the group
  multiplication on $G$ extends to a continuous map from $G\times
  S(G)$ to $S(G)$. (In particular  $S(G)$ is a $G$-space.)
\item If $C$ is a compact $G$-space then for any $p\in C$ the map
$g\mapsto g\ccdot p$ from $G$ to $C$ extends uniquely to a continuous map
from $S(G)$ to $C$.
\end{enumerate}
 We refer to the original article \cite{samuel} and a survey \cite{usp} for more details on the
 Samuel compactification.

\medskip

In this section we investigate further properties of $S_G^\mu(M)$ for
a topological group $G$ definable in a first order structure $\CM$. We show
that it has a natural compact topology and under the additional assumption
that every complete $G$-type is definable, property (ii) above
holds if we restrict ourselves to definably separable actions of $G$
on compact spaces.

In the special  case, when {\em every\/}  subset of $G(M)$ is definable in $\CM$, the
space $S_G^\mu(M)$ is exactly the Samuel compactification $S(G)$ (by properties (i) and
(ii) above). A slightly different model theoretic approach to the Samuel
compactification is contained in \cite[Section 2]{GPP}.

If the topology on $G$ is discrete then $S^\mu_G(M)=S_G(M)$. This case has been
already studied in several papers (e.g.  \cite{N}).

\subsection{The topology of $S_X^\mu(M)$}

\noindent{\bf Setting:} We work in the same setting as in Section~\ref{sec:group-actions}.
We fix a first order structure $\CM$ and by definable we always mean
$M$-definable.
We also fix a group $G$ definable in $\CM$. We
assume  that $G$ is a topological group and that it has a basis of open
neighborhoods of $e$ consisting of definable sets.  As in Section~\ref{sec:group-actions} we will denote by $\mu(v)$ the infinitesimal type of $G$.
Here we always write $G$ for $G(M)$.

If  $X$ is a definable $G$-set then   by Claim~\ref{claim:g-action}, the action of
$G$ on $X$ induces an action of $G$ on $S^\mu_X(M)$.  Our first goal is to show that
$S^\mu_X(M)$ has a natural Hausdorff topology. With respect to this topology
$S_G^\mu(M)$ is compact and the group action is continuous, in other words
$S_X^\mu(M)$ is a compact $G$-space (compact $G$-spaces are also called  $G$-flows).

Since the relation $\sim_\mu$ on $S_X(M)$ is induced by a type definable equivalence
relation (see below), we present the basics in a more general setting, still
denoting the underlying set by $X$.

\medskip

Let  $X$ be  a definable set in $\CM$ and $E(x,y)$ an
$X^2$-type over $M$ which defines an equivalence relation on $X(\mathbb U)$ (we call
it {\em a type-definable equivalence relation on $X$\/}). We denote by $E^*$ the
associated equivalence relation on $S_X(M)$:
$$p \, E^* \, q \Leftrightarrow p(x)\,\cup\, q(y)\,\cup\,E(x,y) \,\mbox{ is
consistent}.$$ For $p\in S_X(M)$, let $[p]$ denote its $E^*$-equivalence class.

\begin{defn} We equip the set $S^*_X(M)=S_X(M)/E^*$ with the quotient topology, namely a subset $F\subseteq S^*_X(M)$ is
closed if and only if there is a partial $X$-type $\Sigma$ over $M$ such that
\begin{equation}\label{eq-top} F=\{ [p]\in S_X^*(M) \colon p\vdash
\Sigma\}.\end{equation}
\end{defn}

\begin{claim}
  \label{claim:mu-haus} The relation $E^*$ is closed in $S_X(M)\times S_X(M)$,  hence
  $S_X^*(M)$ is compact and the projection map $\pi\colon S_X(M)\to
  S_X^*(X)$ is closed.
\end{claim}
\begin{proof} Take $(p,q)\notin E^*$. Then $p(x)\cup q(y)\cup E(x,y)$ is
inconsistent, and therefore by compactness there are formulas $\varphi(x)\in p$ and
$\psi(y)\in q$ such that $\{\varphi(x),\psi(y)\} \cup E(x,y)$ is inconsistent. The
formulas $\varphi$ and $\psi$ define an open neighborhood $U_\varphi \times U_\psi$
of $(p,q)$ which is disjoint from $E^*$, so $E^*$ is closed.

We now conclude that $S^*_X(M)$ is Hausdorff, so by continuity of $\pi$ it is
compact, and $\pi$ is a closed map.\end{proof}

As a corollary of the above we see that for each $p\in S_X(M)$, the
$E^*$-equivalence class $[p]\sub S_X(M)$  is a closed subset of $S_X(M)$, so given
by a partial type, which we denote by $\Phi_p(x)$. The next claim describes the
topology on $S^*_X(M)$.

\begin{claim}
\label{claim:mu-top} Assume that $X$ and $E$ are as above.
\begin{enumerate}[leftmargin=*]
\item A non-empty subset $F\subseteq S_X^*(M)$ is closed if and only if there is
  an $X$-type $\Sigma(x)$ over $M$ such that $$F=\{[p]\in S_X^*(M)
  \colon p(x) \, \cup  \, E(x,y) \cup\, \Sigma(y) \mbox{ is consistent} \}.$$
\item A basis for the open sets in the topology of $S^*_X(M)$ is the collection, as
$\varphi$ varies in $\mathcal L_X(M)$, of sets of the form $$U^*_\varphi:=\{[p]\in
S^*_X(M)\colon  \Phi_p(x)\vdash \varphi(x)\}.$$

\end{enumerate}
\end{claim}
\proof (1) follows from the description of closed sets in (\ref{eq-top}).

\medskip

(2) If $F$ is any closed set as in (1) then its complement is $$\{[p]\in S_X^*(M)\colon
p(x)\cup  E(x,y) \cup \Sigma(y) \mbox{  is inconsistent}\}, $$ which is the same as
$$\{[p]\in S^*_X(M)\colon  \Phi_p(x) \,\cup \, \Sigma(x) \mbox{ is
inconsistent}\}.$$ By compactness, this is the union over all $\varphi\in \Sigma$,
of the sets
$$\{[p]\in S_X^*(M)\colon  \Phi_p\vdash \neg\varphi\}.$$ Every such set is an open set of
the form $U^*_{\neg\varphi}$.\qed

\begin{rem} Note that the notion of a type definable equivalence relation $E$ on $X$ is the
same that that of {\em a uniformity on $X$\/} (see \cite[Chapter 6]{Kelley}). Namely,
the set $\{\phi(M^2)\colon \phi(x,y)\in E\}$ is a base for a uniformity on $X$.
Conversely, given a uniformity $\mathcal U$ on a set $X$, if we endow $X^2$ with a
predicate  for every set in $\mathcal U$, then the set of these predicates is a
type-definable equivalence relation on $X$.

The space $S^*_X(M)$ that we described above is basically the compactification of a
uniform space (a set together with a uniformity) described by Samuel in
\cite{samuel}.
\end{rem}

\medskip

We now return to the case when $X$ is  a definable $G$-space.
Obviously,
$$\Phi(x,y)=\{\exists z(\theta(z)\wedge z\ccdot x=y)\,\colon \, \theta(v)\in \mu\}$$
defines the equivalence relation $\mu(\mathbb U)\ccdot x=\mu(\mathbb U)\ccdot y$ on
$X(\mathbb U)$, and we have  $$p \sim_\mu q\Leftrightarrow p(x)\cup q(y) \cup
\Phi(x,y)\,\, \mbox{ is consistent},$$ therefore $\sim_\mu$ is the  equivalence
relation on $S_X(M)$  associated to $\Phi$, whose quotient space is $S_X^\mu(M)$.

By our above analysis, the space $S_X^\mu(M)$ is compact and a basis for its
topology is given by open sets of the form $$U^\mu_\varphi=\{p_\mu\in
S^\mu_X(M)\colon \mu\ccdot p\vdash \varphi\},$$ as $\varphi$ varies over the $X$-formulas.

Although parts of what we do below can be still presented in a more general setting
we stick to the case of definable $G$-sets.

\begin{claim}
The action of $G$ on $S^\mu_X(M)$ (see Claim~\ref{claim:g-action}) is continuous,
hence
  $S^\mu_X(M)$ is a $G$-space.
\end{claim}

\proof  It is sufficient to show that for each $X$-formula $\varphi$ over
  $M$  the set $W_\varphi=\{ (g,p_\mu) \colon g\ccdot p_\mu \in U_\varphi^\mu\}$ is
  open in the product topology on $G\times S_X^\mu(M)$.

Assume  $(g,p_\mu)\in W_\varphi$. Then there is $\theta(v)\in \mu$ and $\psi(x)\in
p(x)$ such that $(g\ccdot \theta \ccdot \psi)(x) \vdash \varphi(x)$.  Since
$\mu\ccdot\mu=\mu$, there is $\theta_1(v)\in \mu$ such that
$(\theta_1\ccdot\theta_1)(v) \vdash \theta(v)$.  Let $\Theta_1=\{ h\in G \colon
\CM\models \theta_1(h) \}$. It is an open subset of $G$. For every $h\in \Theta_1(M)$ we
have
\[ \bigl((g\ccdot h)\ccdot (\theta_1\ccdot\psi)\bigr)(x)\vdash \varphi(x). \]
It is easy to see that $g\ccdot \Theta_1\times
U^\mu_{\theta_1\ccdot\psi}$ is an open subset of $W_\varphi$
containing $(g,p_\mu)$.
\qed

\medskip

Recall that in general when a group $G$ acts on sets $X$ and $Y$ then
a map $f\colon A\to B$ is called {\em equivariant\/} if it commutes with
the action of $G$, i.e. $f(g\ccdot a)=g\ccdot f(a)$ for all $g\in G$
and $a\in A$.

Also recall that if $f\colon X\to Y$ is an $M$-definable map between $M$-definable
sets (and hence also from $S_X(M)$ to $S_Y(M)$) then $f$ has a canonical extension
to a map $ f_*$ from $\mathcal L_X(M)$ to $\mathcal L_Y(M)$ so that for a saturated
elementary extension $\UU$ of $\CM$ we have $f(\Sigma(\UU))= f_*(\Sigma)(\UU)$ for
every $X$-type $\Sigma(x)$, and if $p\in S_X(M)$ then $f_*(p)\in S_Y(M)$.

\begin{claim}
  \label{claim:morph-ext} Let $X,Y$ be definable $G$-sets and let $f\colon X\to Y$ be an $M$-definable
  equivariant map. Then the map $f_*\colon S_X(M)\to  S_Y(M)$
  respects  $\sim_\mu$, namely for $p\in S_X(M)$ we have $f_*(\mu\ccdot
  p)=\mu\ccdot f_*p$. The induced map from $S_X^\mu(M)$ to $S_Y^\mu(M)$, still
  denoted by $f_*$,
is equivariant and continuous.
\end{claim}
\begin{proof}
  Since $f\colon X \to Y$ is definable, the map from $X(\UU)$ to
  $Y(\UU)$ that it defines is equivariant with respect of the action of
  $G(\UU)$. Using Remark~\ref{rem:products} and properties of $f_*$ we obtain
  \begin{multline*}
  f_*(\mu\ccdot p)(\UU) =   f\bigl((\mu\ccdot p)(\UU)\bigr )=
f\bigl(\mu(\UU)\ccdot p(\UU)\bigr ) \\
= \mu(\UU)\ccdot f\bigl(p(\UU)\bigr)=\mu(\UU)\ccdot f_*(p)(\UU)= (\mu\ccdot
f_*(p))(\UU).
  \end{multline*}
Hence $f_*(\mu\ccdot p) = \mu\ccdot f_*(p)$, and it is not hard to see that the
induced map from $S^\mu_X(M)$ to $S^\mu_Y(M)$ is equivariant.   It is continuous
since $f_*\colon S_X(M)\to S_Y(M)$ is continuous in logic topology.
 \end{proof}

\begin{ntn}
  For a definable $G$-set $X$ we will denote by $\rchi_X$ the map from
  $X$ to $S_X^\mu(M)$ defined as  $\rchi_X\colon a\mapsto
\tp(a/M)_\mu$.
\end{ntn}

\begin{claim}
  \label{claim:ctab-closure}  Let $X$ be a definable $G$-set.

\begin{enumerate}[leftmargin=*]
\item The image $\rchi_X(X)$ is dense in $S_X^\mu(M)$. \item Let $a,b\in X$. Then
$\rchi_X(a)=\rchi_X(b)$ if and only if $b\in \bbar{G_a}\ccdot a$, where $G_a=\{g\in
G \colon g\ccdot a =a\}$ is the stabilizer of $a$ in $G$ and $\bbar{G_a}$ is its
topological closure in $G$. In particular the stabilizer of every $a\in X$ is closed
in $G$ if and only if $\rchi_X$ is injective.
\end{enumerate}
\end{claim}
\begin{proof} (1) Since the set $\{ \tp(a/M)\colon a\in X\}$ is dense
  in $S_X(M)$ in the logic topology, its image in $S_X^\mu(X)$ under the projection map is dense
  as well.

\medskip

(2)  Assume first that $b\in \bbar{G_a}\ccdot a$ and we will show that $\tp(b/M)\vdash
  \mu\ccdot \tp(a/M)$. Choose  $g\in \bbar{G_a}$ with $b=g\ccdot a$.
Since $g$ is in the closure of $G_a$,  for every $\theta(v)\in \mu(v)$ there is $h_\theta\in G_a$ with $g\in \theta(M)\ccdot
 h_\theta$. Thus $b\in \theta(M)\ccdot h_\theta\ccdot a =\theta(M)\ccdot a$.
  Hence $\tp(b/M)\vdash \theta\ccdot \tp(a/M)$ for every $\theta(x)\in
  \mu(x)$,  so $\tp(b/M)\vdash \mu\ccdot \tp(a/M)$

For the opposite direction assume $\mu\ccdot
\tp(b/M)=\mu\ccdot\tp(a/M)$, or equivalently,
$\tp(b/M)\vdash\mu\ccdot\tp(a/M)$.
First we claim that in this case $b$ is in the $G$-orbit of
$a$. Indeed take any $\varphi(v)\in \mu(v)$. We have $\tp(b/M)\vdash
\varphi\ccdot \tp(a/M)$ hence $b=g\ccdot a$ for some $g\in
\varphi(M)$.  Now we fix some $g\in G$ such that $b=g\ccdot a$. For
every $\theta(v)\in \mu(v)$ there is $h_\theta\in \theta(M)$ such
that $b=h_\theta\ccdot a$, hence $h^{-1}_\theta g\in G_a$. Since
$\mu=\mu^{-1}$, and the sets $\theta(M)\ccdot g, \theta(v)\in \mu$ form a
basis of open neighborhoods  of $g$ we obtain that $g$ is in the
closure of $G_a$.
\end{proof}

Thus if the stabilizer of each point $a\in X$ is closed, the map $\rchi_X$ is
injective and we can consider $X$ as a subset of $S_X^\mu(M)$. Notice that this is
indeed the case when $X$ is a definable $G$-set in an o-mnimal structure, since all
definable subgroups of $G$ are closed.

The next claim describes the induced topology on $X$.

\begin{claim}
  \label{claim:orbits-top}
Let $X$ be a definable $G$-space. Assume the map $\rchi_X\colon X\to S_X^\mu(M)$  is
injective. Then after identifying $X$ with $\rchi(X)$ the topology on $X$ induced by
$S_X^\mu(M)$ is exactly the topology whose basis is:
$$\{V\ccdot a: V \mbox{ open in } G\, , \,\, a\in X\}.$$
In particular,
\begin{enumerate}[leftmargin=*]
\item Every $G$-orbit in $X$ is open and closed.
\item For every $a\in X$ the orbit $G\ccdot a$ is homeomorphic to the factor
  space $G/G_a$ (with the quotient topology) under the natural map
  $gG_a\mapsto g\ccdot a$.
\end{enumerate}
\end{claim}
\begin{proof} Let us see that every such $V\ccdot a\sub X$ is indeed open in the induced topology. Without loss of generality,
$V$ is  a definable  open set given by a $G$-formula $\varphi(v)$. We claim that
$U^\mu_{\varphi\ccdot a}\cap \rchi_X(X)=\rchi_X(V\ccdot a)$. Indeed, if
$\tp(b/M)_\mu\in U^\mu_{\varphi(x)\ccdot a}$ then in particular, $b\in V\ccdot a$.
For the converse, if $b\in V\ccdot a$ then there is $g\in V$ with $b=g\ccdot a$. But
then there is an $M$-definable open neighborhood $V_1$ of $g$ contained in $ V$, so
in particular, $(\mu\ccdot g)\ccdot a\vdash V\ccdot a$, hence $\tp(b/M)_\mu\in
U^\mu_{\varphi\ccdot a}$.

Let us see that every open subset of $\rchi_X(X)$ is a union of sets of the form
$\rchi_X(V\ccdot a)$. Consider $ \tp(a/M)_\mu \in U^\mu_{\psi}\cap \rchi_X(X)$ for
some $X$-formula $\psi(x)$ and $x\in X$. Since $\mu\ccdot a \vdash \psi$ there
exists $\theta \in \mu$ with $\theta \ccdot a\vdash \psi$. If we take $V=\theta(G)$
then $\rchi_X(V\ccdot a)\sub  U^\mu_{\psi}$, so we can write $U^\mu_{\psi}\cap
\rchi_X(X)$  as a union of sets of this form.

(1) and (2) easily follow.\end{proof}

\begin{rem}
  It is not hard to see that in the previous claim properties (1)
  and (2) define unique topology on $X$ and it is the strongest
  topology on $X$ making the action of $G$ on $X$ continuous.
\end{rem}

\begin{cor}\label{cor:G-mod-closed}
 If $H$ is a
 definable closed subgroup of $G$ then the space $X=G/H$ with the
 quotient topology embeds into  the compact space $S_X^\mu(M)$ and the
 action of $G$ on $S_X^\mu(M)$ is the unique continuous extension of the
 action of $G$ on $G/H$.

In particular, for $H=\{e\}$ we obtain that the map $\rchi_G$ is a
topological embedding of $G$ into $S_G^\mu(M)$ and under this
embedding the action of $G$
on $S_G^\mu(M)$ is the unique continuous extension of the action of $G$ on itself by left
multiplication.
\end{cor}

\subsection{Definably-separable actions}
\label{sec:acti-defin-topol}

In \cite{GPP} a map $f$ from a definable set $D$ to a compact space
$C$ was called definable if for every disjoint closed
$C_1,C_2\subseteq C$ their prei-mages $f^{-1}(C_1)$ and $f^{-1}(C_2)$ can be separated by definable subsets of $D$.
Since  we prefer
to reserve the term ``definable'' for actual definable maps, we
will use the term ``definably separable'' instead.

As in \cite{GPP} we say that  an action of a definable group
  $G$ on a compact space $C$ is {\em definably separable\/}  if
  for every $x_0\in C$ the map from $G$ to $C$ given by $g\mapsto
  g\ccdot x_0$ is definably separable.

\begin{lem} Assume $G$ acts continuously on a compact space $C$.  Let
  $c_0\in C$ and assume the map $f\colon G\to C$ given by $g\mapsto
  g\ccdot c_0$ is definably separable.  Then $f$ can be extended
  uniquely to a continuous $G$-equivariant map  $f_*\colon S_G^\mu(M)\to C$.
\end{lem}
\begin{proof} For $C_0\subseteq C$ we will denote by $\bbar{C_0}$ the
  topological closure of $C_0$ in $C$.

The definition of the map is classical and goes back to Stone-\v{C}ech: If $p(v)\in
S_G(M)$ then using definable separation of $f$ and compactness of $C$ it follows
that the set
\[ f[p]: =\bigcap_{p(v)\vdash \varphi(v)} \bbar {f(\varphi(M)) }\] is a
singleton. This gives a map from $S_G(M)$ to $C$. We claim that for
$p,q\in S_G(M)$ with $p\sim_\mu q$ we have $f[p]=f[q]$.

Assume not, and $f[p]\neq f[q]$ for some $p\sim_\mu q$, hence
$f[p]\cap f[q] =\emptyset$.  By compactness of $C$ it
follows then that there are $\varphi(v)\in p(v)$ and $\psi(v)\in q(v)$
such that $\bbar{f(\varphi(M) )} \cap \bbar{f(\psi(M) )} =\emptyset$.

In general, by standard compactness arguments, when a group $G$ acts
continuously on a compact space $C$, for any two disjoint closed subsets $C_0,
C_1$ of $C$, there is an open subset $\CO$ of $G$ containing $e$ such
that $\CO\ccdot C_1 \cap C_2=0$. Therefore there is a formula
$\theta(v)\in \mu$ such that
\[ \theta(M)\ccdot \bbar{f(\varphi(M) )} \cap \bbar {f(\psi(M) )}
=\emptyset, \]
and in particular $\theta(M)\ccdot {f(\varphi(M) )} \cap {f(\psi(M) )}
=\emptyset$.

Therefore $(\theta(M)\ccdot \varphi(M))\ccdot c_0 \cap
\psi(M)\ccdot c_0
=\emptyset$, and hence $\theta(M)\ccdot \varphi(M) \cap
\psi(M)=\emptyset$.
But this contradicts to consistancy of $\mu\ccdot p$ and $q$.

For $p\in S_G(M)$ we define $f_*(p_\mu)$ to be the unique element in
$f[p]$. It is not hard to check that the map $f_*$ is continuous.

The uniqueness of $f_*$ and its $G$-equivariance follow from the density of $G$ in $S_G^\mu(M)$.

 \end{proof}

The proof of the following claim is similar to the proof of
\cite[Lemma 3.7]{GPP}

\begin{lem}
  \label{claim:orb-def-sep} Let $X$ be a definable $G$-set, and assume
$p\in S_X(M)$ is a definable type. Then the map $f_p\colon G\to
S_X^\mu(M)$ given by $g\mapsto g\ccdot p_\mu$ is definably separable.
In particular, if every type $p\in S_X(M)$ is definable then the
action of $G$ on $S_X^\mu(M)$ is definably separable.
\end{lem}
\begin{proof} Let $C_1, C_2$ be disjoint closed subsets of $S_X^\mu(M)$.
 By Claim~\ref{claim:mu-top}(1) there are  $X$-types $\Sigma_1(x),
 \Sigma_2(x)$ such that
\[ C_i=\{ s_\mu(x) \in S^\mu_X(M) \colon s\vdash \mu\ccdot
 \Sigma_i\},\quad  i=1,2.\]
It follows that $\mu\ccdot \Sigma_1 \cup \mu\ccdot \Sigma_2$ is inconsistent
so there are
 $\theta(v)\in \mu$ and $\psi_i(x)\in \Sigma_i(x)$ such that
the $(\theta\ccdot \psi_1 \wedge \theta\ccdot \psi_2)(x)$ is
inconsistent.  Obviously we have $q_\mu\in C_i \Rightarrow q(x)\vdash
(\theta\ccdot \psi_i)(x)$.

Since $p(x)$ is a definable type the sets $U_i=\{ g\in G \colon
g\ccdot p(x) \vdash   (\theta\ccdot \psi_i)(x)\}$, $i=1,2$, are definable. They are
disjoint and, by above, $g\ccdot p_\mu\in C_i \Rightarrow g\in U_i$,
hence $f_p^{-1}(C_i)\subseteq U_i$.
\end{proof}

\begin{rem}  It follows from the proof that instead of the assumption
  of definability of $p$ it is sufficient to assume that for any
  formulas $\varphi(x)$ the set $\{g\in G \colon g\ccdot p \vdash
  \varphi(x) \}$ is definable.
\end{rem}

In the next theorem we list main properties of $S_G^\mu(M)$.
\begin{thm}
  \label{thm-def-sam}
  \begin{enumerate}[label=(\roman*), leftmargin=*]
  \item
The space $S_G^\mu(M)$ is compact. There is a
  topological
  embedding $\rchi\colon G\to S_G^\mu(M)$ such that the image
  $\rchi(G)$ is dense in $S_G^\mu(M)$, and the group multiplication of
  $G$ extends to a continuous map from $G\times S^\mu_G(M)$ to $S_G^\mu(M)$.
\item Assume that every type $p\in S_G(M)$ is   definable. Then the
  action of $G$ on $S^\mu_G(M)$ is definably separable.
\item For any
  compact $G$-space $C$ and $p\in C$ if the map $g\mapsto g\ccdot p
  $ is  definably separable then it  extends
  uniquely to a continuous $G$-equivariant map from $S^\mu_G(M)$ to $C$.
  \end{enumerate}
\end{thm}

\begin{rem}
 \begin{enumerate}[label=(\alph*),leftmargin=*]
  \item
It follows from the above theorem that in the case when
  every type $p\in S_G^\mu(M)$ is definable, the space
  $S_G^\mu(M)$ has the same universal property as the Samuel
  compactification if one considers only definably separable actions
  of $G$ on compact spaces.

\item In general, the spaces $S_G^\mu(M)$ and the Samuel compactification
$S(G)$ have very different properties. For example, by a theorem of W.~Veech \cite{V} every
locally compact group acts freely on its Samuel compactification, but
as it follows from Theorem~\ref{theorem:1} if $H$ is a torsion free
group definable in an o-minimal expansion $\CM$ of the real field
then $H$ fixes a type $p_\mu\in S_H^\mu(M)$.

\item Theorem~\ref{theorem:1} and Fact~\ref{Borel} give a complete
  description of  minimal compact $G$-invariant subsets of $S_G^\mu(M)$
  in the case when $G$ is a group definable in an o-minimal expansion
  of the real field.  They are exactly orbits of $\mu$-types $q_\mu$
  for types $q$ whose stabilizers are maximal torsion free subgroups
  of $G$. Indeed, each such orbit is compact since $G/\Smu(q)$ is compact. If $Y\sub
  S^\mu_G(M)$ is a minimal compact $G$-invariant set then there exists a $G$-map
  $f\colon S^\mu_G(M)\to Y$ which, by minimality, necessarily sends the orbit of $q_\mu$ above
  onto $Y$. It follows that $Y=G\ccdot f(q_\mu)$ and $\Stab(q_\mu)\sub
  \Stab(f(q_\mu))$. Because $\Smu(q)$ is a maximal torsion-free group,
  $\Stab(q_\mu)=\Stab(f(q_\mu))$.

\item As for Samuel compactifications, it is easy to see, by
  properties (ii) and (iii) in Theorem~\ref{thm-def-sam}, that when  every type $p\in S_G^\mu(M)$ is definable,   the
  action of $G$ on $S_G^\mu(M)$ extends to a semi-group operation on $S_G^\mu(M)$.

  \end{enumerate}
\end{rem}

\renewcommand{\thesection}{Appendix~\Alph{section}}
\section{On connected components of relatively definable subsets}
\label{sec:setting}
\renewcommand{\thesection}{\Alph{section}}

We work in an o-minimal structure  $\CM$  and  by definable we mean definable in
$\CM$.

Recall that every definable non-empty subset $X\subseteq M^n$ is a disjoint union of
finitely many definably connected components $X_1,\dots X_k$.  These components are
unique, up to a permutation, and characterized by the following two properties:
\begin{enumerate}[label=(\Roman*),leftmargin=*]
\item  Every $X_i$ is a non-empty definable subset of $X$ that is
  clopen  in
$X$. \item If $Y\subseteq X$ is a definable clopen set and $Y\cap
  X_i\neq \emptyset$ then $X_i\subseteq Y$.
\end{enumerate}

In this section we show that the same is true for  relatively definable subsets of
convexly  definable open sets

\medskip
\def\CX{\mathcal{X}}
\def\CJ{\mathcal{J}}

Let $I\subseteq M$ be an interval and $\{V_i\colon i\in I\}$ a uniformly definable
family of open sets of $M^n$.  For a closed downward subset $\CJ\subseteq I$ we will
denote by $V_\CJ$ the open set
\[ V_\CJ=\bigcup_{ r\in \CJ} V_r, \]
and call such an open set {\em convexly definable}.

As usual, a subset $\CX \subseteq V_\CJ$ is called {\em relatively definable\/} if
$\CX=X\cap V_\CJ$ for some definable set $X$.

\begin{sample}
\label{sample:conv-def} Assume $\CM$ is an o-minimal expansion of a group and
$\CM_0\esub \CM$ is an elementary substructure. Then the convex hull $\CO_{M_0}(M)$
and all its cartesian powers are convexly definable open sets, via $V_r=(-r,r)$, for
$r\in M_{\geq 0}$.
\end{sample}

The following is the main theorem of this section.

\begin{thm}
  \label{thm:conv-def} Let $V_\CJ\subseteq M^n$ be a convexly
  definable open set and  $\CX \subseteq V_\CJ$  a relatively definable set. Then
there are disjoint non-empty relatively definable subsets $\CX_1,\dotsc,\CX_k
\subseteq V_\CJ$ such that $\CX=\cup_{i=1}^k \CX_i$ and
\begin{enumerate}[leftmargin=*]
\item  Every $\CX_i$  is clopen in $\CX$ in the topology
  induced from $M^n$;
\item If $\mathcal Y\subseteq \CX$ is a relatively definable clopen set and $\mathcal Y\cap
  \CX_i\neq \emptyset$ then $\CX_i\subseteq \mathcal Y$.
\end{enumerate}
\end{thm}

We call the sets $\CX_i$ above {\em the connected
  components of $\CX$}.

In the rest of this section we prove the above theorem.

\medskip
Replacing $V_r$ by $\cup_{s\leq r}V_s$ if needed we assume that $s\leq r$ implies
$V_s\subseteq V_r$. Let $X \subseteq M^n$ be a definable set such that $\CX=X\cap
V_\CJ$. For $r\in I$ we will denote by $X_r$ the set $X\cap V_r$, so $\CX=\cup_{r\in
\CJ} X_r$.

\medskip

For $r\in I$ and $\alpha\in X_r$ we will denote by $X_r(\alpha)$ the definable
connected component of $X_r$ containing $\alpha$.

The following Claim follows easily from o-minimality and properties of connected
components.
\begin{claim}
  \label{claim:uniform}
\begin{enumerate}[leftmargin=*]
\item The family $\{ X_r(\alpha) \colon \alpha\in X, r\in I\}$ is uniformly
definable. \item For any $\alpha\in M^n$ and  $r_1<r_2 \in I$ we have
\[ X_{r_1}(\alpha)\subseteq X_{r_2}(\alpha). \]
\item Assume  $X_r(\alpha)\cap X_r(\beta)\neq \emptyset$. Then
$X_r(\alpha)=X_r(\beta)$ and for all $r<s\in I$  we have   $X_s(\alpha)=X_s(\beta)$.
\end{enumerate}
\end{claim}

For $\alpha\in \CX$ we define
\[ \CX(\alpha)=\bigcup_{r\in \CJ} X_r(\alpha). \]

The following claim follows from Claim~\ref{claim:uniform}(3).
\begin{claim}\label{claim:disjoint}
  For $\alpha,\beta \in \CX$, either $\CX(\alpha)=\CX(\beta)$ or $\CX(\alpha)\cap \CX(\beta)=\emptyset$.
\end{claim}

For $\alpha\in \CX$ let's call the  set $\CX(\alpha)$ {\em a
 component of $\CX$}.

\begin{claim}\label{claim:fin-many-comp}  $\CX$ has finitely many components.
\end{claim}
\begin{proof}
By o-minimality there an integer $N$ such that for every $r\in I$ the set $X_r$ has
at most $N$ connected components.  We claim that $\CX$  has at most $N$  components.

Assume not. Then there are $\alpha_0,\dotsc \alpha_N\in \CX$  such that the sets
$\CX(\alpha_i), i\leq  N$ are disjoint.   We can choose $r\in \CJ$ such that all
$\alpha_i$ are in $X_r$.  But then $X_r$ has at least $N+1$ connected components.  A
contradiction.
\end{proof}

\begin{claim}\label{claim:comp-def}
Each component $\CX_i$ is relatively definable.
\end{claim}
\begin{proof}
  The claim is obvious if $\CJ$ has a least upper bound (in
  $M\cup\{+\infty\}$),
since then $\CX$ and all $\CX_i$ are definable. Assume $\CJ$ does not have  a least
upper bound. Choose $\alpha_1,\dotsc\alpha_k\in \CX$ such that
$\CX_i=\CX(\alpha_i)$.

For every $i< j$  consider the set of all $r\in I$ such that $X_r(\alpha_i)\cap
X_r(\alpha_j)=\emptyset$.  By Claim~\ref{claim:uniform}(1), it is a definable set
containing $\CJ$, hence contains an element $r_{ij}\in I$ that is not in $\CJ$.

Let $r^*=\min\{ r_{ij} \colon i<j \}$. Notice $r^*\not\in \CJ$.  Since the set
$X_{r^*}(\alpha_i)$ is definable and $\CX_i=X_{r^*}(\alpha_i)\cap V_\CJ$, the set
$\CX_i$ is relatively definable.
\end{proof}

\begin{claim} \label{claim:clopen} For every $\alpha\in \CX$, the component $\CX(\alpha)$ is clopen
 in $\CX$, with respect to the $M^n$-induced topology.
\end{claim}
\begin{proof} Since $\CX$ is a disjoint union of finitely many
  $\CX(\alpha)$ we only need to show that each $\CX(\alpha)$ is open.

Since every $V_r$ is open, for any $r\in \CJ$ the set $X_r=\CX\cap V_r$ is open in
$\CX$.  The connected component $X_r(\alpha)$ of $X_r$ is open in $X_r$. Hence
$X_r(\alpha)$ is open in $\CX$, and $\CX(\alpha)$ is open in $\CX$ as a union of
open sets.
\end{proof}

\begin{claim}
\label{claim:prop2} Let $\mathcal Y$ be a relatively definable clopen
subset of $\CX$. If $ \mathcal Y\cap \CX_i\neq \emptyset$ then $\CX_i \subseteq
\mathcal Y$.
\end{claim}
\begin{proof}
 Assume  $\mathcal Y\cap \CX_i\neq \emptyset$ and  choose $\alpha\in
 \mathcal Y\cap \CX_i$. Since $\mathcal Y$ is relatively definable subset,
 for every $r\in \CJ$ the set $\mathcal Y\cap V_r$ is a definable subset of $X_r$
 that is clopen in $X_r$. Thus it contains the definably
 connected component $X_r(\alpha)$, hence $\mathcal Y$ contains
 $\CX_i=\CX(\alpha)$.
\end{proof}

This finishes the proof of Theorem~\ref{thm:conv-def}

\bigskip
\footnotesize
\noindent\textit{Acknowledgments.}
The authors thank Anand Pillay for useful discussions and in particular for his
alternative proof of the definability of $\Smu(p)$, which gave rise to Proposition~\ref{DCC}.
They also thank the referee for very useful comments on the earlier
version of the paper. 

The authors thank  the US-Israel Binational Science Foundation for its
suppor, and also thank
the Mathematical Science Research Institute at Berkeley  for its hospitality during
Spring 2014. The second author thanks the National Science Foundation for its support.

\end{document}